\documentclass[12pt]{amsart}
\usepackage{amsmath}
\usepackage{mathtools}
\usepackage{amscd}
\usepackage{amsfonts}
\usepackage[foot]{amsaddr}
\usepackage{amsthm}
\usepackage{mathrsfs}
\usepackage{enumerate}
\usepackage{amssymb}
\usepackage{color}

\usepackage[dvipdfmx]{hyperref}

\theoremstyle{plain}
\newtheorem{proposition}{Proposition}[section]
\newtheorem*{proposition*}{Proposition}
\newtheorem{lemma}[proposition]{Lemma}
\newtheorem*{proof*}{proof}
\newtheorem{corollary}[proposition]{Corollary}

\newtheorem*{definition*}{Definition}
\newtheorem{theorem}[proposition]{Theorem}
\theoremstyle{remark}
\newtheorem{remark}{Remark}[section]
\makeatletter
\@addtoreset{equation}{section}

\makeatother

\newcommand{\Z}{\mathbb{Z}}
\newcommand{\Q}{\mathbb{Q}}
\newcommand{\R}{\mathbb{R}}
\newcommand{\C}{\mathbb{C}}
\renewcommand{\H}{\mathcal{H}}

\newcommand{\SL}{\mathrm{SL}}
\newcommand{\GL}{\mathrm{GL}}

\newcommand{\Mp}{\mathrm{Mp}}

\newcommand{\OO}{\mathrm{O}}

\newcommand{\bp}{{b^+}}
\newcommand{\bm}{{b^-}}
\newcommand{\tp}{{t_+}}
\newcommand{\tm}{{t_-}}
\newcommand{\sigmap}{{\sigma_+}}
\newcommand{\sigmam}{{\sigma_-}}

\newcommand{\be}{\mathbf{e}}

\newcommand{\sig}{\mathrm{sig}}

\title[Theta lifts for indefinite orthogonal groups]
{Theta lifts to certain cohomological representations of indefinite
orthogonal groups}
\author[Takuya Miyazaki]{Takuya Miyazaki}
\author[Yohei Saito]{Yohei Saito}
\address{Department of Mathematics, Faculty of Science and Technology, Keio University, Hiyoshi, Yokohama 223-8522, Japan}
\email{miyazaki@math.keio.ac.jp}
\email{yohei1031@keio.jp}
\begin{document}

\maketitle

\begin{abstract}
 Howe and Tan \cite{HT} investigated a degenerate principal series 
representation of indefinite orthogonal groups $\OO(\bp,\bm)$ and 
explicitly described its composition series. 
They showed that there exists a unique unitarizable irreducible 
submodule $\Pi$, which is isomorphic to a cohomological representation.
In this paper we construct orthogonal automorphic forms belonging to 
$\Pi$ as theta liftings of holomorphic 
$\Mp_2(\R)$ cusp forms by using the Borcherds' method \cite{Bo}. 
We will give a useful observation 
in computing their Fourier expansions that enables us to give 
very precise descriptions of the liftings specially
belonging to $\Pi$.
We will also determine the cuspidal supports of the liftings \cite{MW}, then
confirm the square integrability of those automorphic forms on 
orthogonal groups.
\end{abstract}

\section{Introduction}

Let $V_\Q$ be a vector space over $\Q$ with a quadratic form $Q$ of signature
$(b^+,b^-)$ on $V=V_\Q\otimes\R$ with $b^+b^->1$. 
Let $\OO(V)$ be the orthogonal group.
We take a maximal parabolic subgroup $P$ in $\OO(V)$,
which stabilizes an isotropic line in $V$.
Howe and Tan \cite{HT} investigated
a degenerate principal series representation
$I_\varepsilon(\lambda+\rho)$ of $\OO(V)$ 
induced from $P_\R$ with $\lambda\in \C$, $\varepsilon=\pm$, and
$\rho=\frac{b^++b^-}{2}-1$ (see Section 2 in this paper).
They gave the complete description of its composition series.
In particular when $k=\lambda-\frac{b^+}{2}+\frac{b^-}{2}+1\in \mathbb{N}$
is large enough and $\varepsilon=(-1)^k$, 
there occurs a unique irreducible submodule
$\Pi^{b^+,b^-}_{k,0}$ of $I_\varepsilon(\lambda+\rho)$ that
admits rich properties.
Here we mention two of them:
\begin{itemize}
 \item The underlying $(\frak g, K)$-module of 
$\Pi^{b^+,b^-}_{k,0}$ is identified with
a derived functor module $A_\frak q(\lambda-\rho)$ 
(see \cite[Section 2]{Ko21}). 
Here the Levi subalgebra of $\frak q$ corresponds to a subgroup of 
$\OO(V)\simeq\OO(\bp,\bm)$ isomorphic to $\OO(2)\times\OO(b^+-2,b^-)$.
\item The derived functor module, hence $\Pi^{b^+,b^-}_{k,0}$,
corresponds to a discrete series representation 
of $\Mp_2(\R)$ in the setting of theta correspondence for the
reductive dual pair
$\Mp_2(\R)\times\mathrm{O}(b^+,b^-)$ (see \cite[Section 6]{Li}
and \cite[Section 6]{BMM}).
\end{itemize}
Exploiting the latter property we can construct, at least in principle,
modular forms on $\OO(V)$ belonging to $\Pi^{\bp,\bm}_{k,0}$ 
as the theta liftings 
relevant to an even lattice $L$ of signature $(b^+,b^-)$.
Actually they can be realized as 
liftings of vector-valued modular forms on $\Mp_2(\R)$
by using Borcherds' method \cite{Bo}.

Let $\H^m(\R^\bp)$ be the space of all homogeneous polynomials on $\R^\bp$
of degree $m$ annihilated by the Laplacian 
$\Delta_\bp=\sum_{j=1}^\bp\partial^2/\partial 
x_j^2$; i.e. the  space of homogeneous spherical harmonic polynomials 
on $\R^\bp$ of degree $m$.
As is well known it gives an irreducible $\OO(\bp)$-module.
Then the $K$-types of $\Pi^{b^+,b^-}_{k,0}$, $K\simeq \OO(b^+)\times \OO(b^-)$,
can be given as
$$
\bigoplus_{m,n\in \mathbb{N},~ m-n\in 2\mathbb{N}+k}
\H^m(\R^\bp)\boxtimes\H^n(\R^\bm).
$$
According to this description it becomes important for us to adjust 
the Borcherds' method so that the constructed theta liftings
exactly belong to the minimal $K$-type 
$\H^k(\R^\bp)\boxtimes\H^0(\R^\bm)=\H^k(\R^\bp)\boxtimes \mathbf{1}$ 
of $\Pi^{b^+,b^-}_{k,0}$.
On this issue we will provide a useful identity
$\vartheta_L(f,g; p)=\vartheta_L(f,g; Hp)$ in Corollary {\ref{hpidentity}}
of Theorem \ref{harmonicprojection}.
While we refer to Section 4 for the notation not explained 
here we point out that this identity will serve to fairly simplify 
the work of explicit computation of the Fourier expansions of 
the theta liftings.
Then our main result on the Fourier expansion will be given in Theorem 
\ref{fourierexpansion>2}, where the case $\bp>2$ is treated.

If $\bp=2$, then the module $\Pi^{2,\bm}_{k,0}$ is a discrete series
representation of $\OO(2,\bm)$. 
Then the closely related theta lifting gives a cusp form on $\OO(2,\bm)$ 
and its Fourier expansion was explicitly described by Oda \cite{Od} and 
Rallis and Schiffmann \cite{RS} (see also \cite[Section 14]{Bo}).
In contrast to this when $\bp> 2$ the module $\Pi^{b^+,b^-}_{k,0}$ 
is not a discrete series representation 
(one can see for instance that the associated local Arthur parameter 
is {\it highly non-tempered}, 
see \cite[Sections 4 and 6]{BMM}). 
This fact contributes to the appearance of constant term in Theorem 
\ref{fourierexpansion>2}.
Our goal is then to establish square integrability of the liftings
on $\varGamma_L\backslash \OO(L_\R)$.
To do that we must specify the cuspidal supports of the liftings, \cite{MW},
by examining the degenerate terms of their Fourier expansions.
We will succeed on this, and confirm the square integrability in 
Theorem \ref{theoremonsquareintegrability} by checking the admissible 
growth condition of the exponents of those cuspidal supports.

Until the end of Section 6, only the $K$-module 
$\H^k(\R^\bp)\boxtimes\mathbf{1}$ is exploited to distinguish
the representation theoretic aspects of the theta liftings.
Thus, strictly speaking, it is not certain that the liftings really belong to 
the $\OO(\bp,\bm)$-module $\Pi^{\bp,\bm}_{k,0}$.
This problem can be completely solved in Section 7
by computing the inverse Fourier
transform of a Bessel function $W_\lambda^\kappa(g)$, 
(\ref{archimedeancoefficient}), applying the method of 
Pollack \cite{Po} and Kobayashi-Mano \cite{KM}.
Put another way, this is the explicit computation of the Bessel 
integral attached to the representation $\Pi^{\bp,\bm}_{k,0}$ of
$\OO(\bp,\bm)$, which is interesting in its own right.
Actually Corollary \ref{exactcharacterization}
confirms that our liftings exactly belong to the module
$\Pi^{\bp,\bm}_{k,0}$ as desired.

It is known that the restriction of the module $\Pi^{b^+,b^-}_{k,0}$ 
to $\mathrm{O}(\bp,\bm-1)$ satisfies a nice discrete 
branching law (see \cite[Theorem 1.1 and Example 1.2]{Ko21}).
In the final Section 8, we are particularly interested in 
the restriction of $\OO(3,2)$-module $\Pi^{3,2}_{k,0}$ to $\OO(3,1)$, 
and discuss the corresponding restrictions of
our liftings with a perspective of the {\it theta contractions} discussed in 
\cite{Ma}, \cite{Zeb}.

The paper is organized as follows.
In Section 2 we recall the results in \cite{HT} 
on module structure of degenerate principal series
representations of orthogonal groups and define 
the irreducible submodule $\Pi^{\bp,\bm}_{k,0}$.
In Section 3 we define our theta lifting following \cite{Bo}.
In Section 4, after reviewing the theory of spherical harmonic polynomials
by Vilenkin \cite{Vi}, we will compute the theta liftings of Poincar\'e series.
Then we will obtain Theorem \ref{harmonicprojection} and its Corollary 
\ref{hpidentity}.
By combining these results with Borcherds' general argument we will compute 
the Fourier expansion of the theta lifting in Section 5.
In Section 6 we will confirm the square integrability of the lifting
$\vartheta_L(f)(g)$ by carefully looking at the constant terms of the expansions
along $\Q$-parabolic subgroups of the orthogonal groups.
Section 7 is devoted to computation of inverse Fourier transforms
of the Bessel functions occurring in the Fourier expansion formula
in Theorem \ref{fourierexpansion>2}.
Then we will establish that $\vartheta_L(f)(g)$ exactly belongs
to $\Pi^{\bp,\bm}_{k,0}$.
Finally in Section 8 we describe especially a restriction to $\OO(3,1)$
of $\vartheta_L(f)(g)$ on $\OO(3,2)$ and discuss a relation with the results
by Zemel \cite{Zeb}.

\section{Degenerate principal series on $\mathrm{O}(\bp,\bm)$}
\label{dps}
In this section we recall the results in \cite{HT}
about the module structures of certain 
degenerate principal series representations
of $\OO(\bp,\bm)$.

Let $\bp$ and $\bm$ are integers such that $\bp\bm>1$.
We set $b=\bp+\bm$.
Denote by $V=\R^{b^+,b^-}$ the vector space over $\R$ with the bilinear form
$(~,~)$ defined by 
\begin{equation}\label{V}
(x,y)=\sum_{j=1}^\bp x_jy_j-\sum_{j=\bp+1}^b x_jy_j
\end{equation}
for $x=\sum_{j=1}^b x_je_j$ and $y=\sum_{j=1}^by_je_j\in V$
with an orthogonal basis $\{e_j\mid 1\le j\le b\}$ of $V$.
Thus $V$ has signature $(\bp,\bm)$.
We take a quadratic form $Q(x)=(x,x)/2$, $x\in V$, on $V$.
Let $\OO(V)$ denote the orthogonal group that leaves invariant
the bilinear form.

Take isotropic vectors
\begin{equation}\label{standardisotropic}
z_0=\frac{e_1+e_b}{2}\ \hbox{ and }\  z_0^\ast=e_1-e_b,
\end{equation}
which satisfy $(z_0,z_0^\ast)=1$.
Thus $\langle z_0, z_0^\ast\rangle$ is a real hyperbolic plane in $V$
and we have the orthogonal decomposition 
$V=\langle z_0, z_0^\ast\rangle\oplus V_1$, where 
$V_1=\langle e_j\rangle_{j=2}^{b-1}=\R^{\bp-1,\bm-1}$.
The group $\OO(V)$ contains the maximal parabolic subgroup $P$
stabilizing the isotropic line $\langle z_0\rangle$.
It admits a Levi decomposition $P=MN$, whose unipotent radical
$N$ is abelian and the Levi subgroup $M$ stabilizes the inclusions
$\langle z_0\rangle\subset \langle z_0, z_0^\ast\rangle\subset V$.
This setup provides group isomorphisms 
\begin{equation}\label{leviisomorphismsV}
n_0: V_1\simeq N\ \hbox{ and }\
m_0: \GL(\langle z_0\rangle)\times \mathrm{O}(V_1)\simeq M.
\end{equation}

For $\alpha\in \C$, and a choice $\pm$ of sign,  
we define the character $\chi_\alpha^\pm$ of $P$ by 
putting 
$$
\chi_\alpha^\pm\bigl(n_0(x)m_0(a,g_1)\bigr)=
\begin{cases} 
a^\alpha,\quad &a>0,\\
\pm |a|^\alpha \quad &a<0.
\end{cases}
$$
Let $I_\pm(\alpha)$ be the corresponding induced representation:
\begin{equation}
I_\pm(\alpha)
=\left\{f\in C^\infty(\OO(V))
\mid f(pg)=\chi_\alpha^\pm(p)f(g)
\hbox{ for } p\in P \hbox{ and } g\in \OO(V)\right\}
\label{Def;principal series}
\end{equation}
on which $\mathrm{O}(V)$ acts by the right translations.

We take the orthogonal decomposition $V=V^+\oplus V^-$ by putting 
$V^+=\langle e_j\rangle_{j=1}^\bp$ and $V^-=\langle e_j\rangle_{j=\bp+1}^b$.
It determines the maximal compact subgroup 
$\OO(V^+)\times\OO(V^-)$ of $\OO(V)$.
We also set $V_1^+=V_1\cap V^+$ and $V_1^-=V_1\cap V^-$.
For every $f\in I_\pm(\alpha)$
the action of $(\OO(V^+)\times\OO(V^-))\cap M=\{\pm 1\}\times \OO(V_1^+)\times 
\OO(V_1^-)$ on $f\vert_K$ factors through the subgroup $\{\pm 1\}$.
Hence the Frobenius reciprocity shows
that the $\OO(V^+)\times\OO(V^-)$-types 
of $I_+(\alpha)$ and $I_-(\alpha)$ are equivalent to
\begin{equation}\label{ktypes}
 \bigoplus_{\substack{m,n\in \mathbb{N}\\ m+n: \,\mathrm{even}}}
\H^m(V^+)\boxtimes\H^n(V^-)\ \
\hbox{ and }\bigoplus_{\substack{m,n\in \mathbb{N}\\ m+n: \,\mathrm{odd}}}
\H^m(V^+)\boxtimes\H^n(V^-)
\end{equation}
respectively, where $\H^l(V^\pm)$ denotes the space of spherical 
harmonic polynomials on $V^+=\R^{\bp}$, or $V^-=\R^\bm$, of degree $l$.

Denote by $X^0\subset V$ the light cone of nonzero isotropic vectors
in $V$;
$$
X^0=\{z\in V-\{0\}\mid (z,z)=0\}.
$$
We see that $\OO(V)$ acts transitively on $X^0$ and
hence $X_0$ is the $\OO(V)$-orbit of the point $z_0\in X^0$, 
(\ref{standardisotropic}).
For any $\alpha\in \C$, we define
$$
S^\alpha(X^0)=\left\{f\in C^{\infty}(X^0) \mid 
f(t z)=t^{-\alpha} f(z),\ t\in \R^\times_{>0},\ z\in X^0\right\}
$$
and its two subspaces
$$
S^{\alpha\pm}(X^0)=\{f\in S^\alpha(X^0)\mid f(-z)=\pm f(z)\}.
$$
Each of these spaces is stable under 
the action of $\mathrm{O}(V)$ defined by
$$
g\cdot f(z)= f(g^{-1}z),\quad z\in X^0,\quad g\in \mathrm{O}(V),
$$
and we get the decomposition
$S^\alpha(X^0)=S^{\alpha+}(X^0)\oplus S^{\alpha-}(X^0)$.
Furthermore if we set $\delta(f)(g)=g\cdot f(z)$ 
for $f\in S^{\alpha\pm}(X^{0})$ and $g\in \mathrm{O}(V)$, 
then we can check that 
$\delta(f)(pg)=f(g^{-1}a^{-1}z)=\chi_\alpha^\pm(p)\delta(f)(g)$
for every $p=n_0(x)m_0(a,g_1)\in P$, as well as that
$\delta(f)(gh)=\delta(h\cdot f)(g)$ for $h\in \mathrm{O}(V)$.
Therefore the map $\delta: S^{\alpha\pm}(X^0)\simeq I_\pm(\alpha)$ 
gives an equivalence of $\mathrm{O}(V)$-modules.

Denote by $\mathbb{S}^{\bp-1}\simeq \OO(V^+)/\OO(V_1^+)$ 
the unit sphere in $V^+=\R^\bp$.
There is the well known decomposition
$$
C^\infty(\mathbb{S}^{\bp-1}\times\mathbb{S}^{\bm-1})\simeq
\sum_{m,n\in \mathbb{N}}\H^m(V^+)\boxtimes\H^n(V^-)
$$
into the sum of inequivalent irreducible $\OO(V^+)\times\OO(V^-)$-submodules.
On the other hand for every $x\in V$ let us write $x=x^++x^-$ 
according to the decomposition $V=V^+\oplus V^-$.
Then we have a diffeomorphism $\beta: X^0\to \mathbb{S}^{\bp-1}\times
\mathbb{S}^{\bm-1}\times\R^\times_{>0}$ defined by
$$
\beta(z)=\left(\frac{z^+}{\Vert z^+\Vert},
\frac{z^-}{\Vert z^-\Vert}, \Vert z^+\Vert\right)
$$
with $\Vert w\Vert=\vert (w,w)\vert^{1/2}$ on each quadratic space $V^\pm$
(see \cite[Scholium 2.1]{HT}).
Combining these ingredients together, we can define an 
$\OO(V^+)\times\OO(V^-)$-equivalent map
$$
j_{\alpha,m,n}=j_\alpha:\H^m(V^+)\boxtimes\H^n(V^-)\to 
S^{\alpha+}(X^0)\oplus S^{\alpha-}(X^0)
$$
by putting
\begin{equation}\label{embedding}
j_\alpha(h_1\otimes h_2)(z)=h_1(z^+)h_2(z^-)\Vert z^+\Vert^{-\alpha-m-n},
\  h_1\in \H^m(V^+),\ h_2\in\H^n(V^-)
\end{equation}
(see \cite[Lemma 2.2]{HT}).
Hence by taking composition with $\delta$, 
any $K$-finite element in $I_\pm(\alpha)$ can be written as 
a linear combination of 
\begin{equation}\label{dpssections}
h_1((g^{-1}z)^+)h_2((g^{-1}z)^-)\Vert (g^{-1}z)^+\Vert^{-\alpha-m-n},
\ g\in\mathrm{O}(V),
\end{equation}
with varying $h_1\in\H^m(V^+)$ and $h_2\in \H^n(V^-)$.
 
Now we recall the following result by Howe and Tan \cite[Sections 2 and 3]{HT} 
(see also \cite[Proposition 2.4]{Ko21}):
\begin{proposition}\label{cohomologicalrep}
Suppose $\bp\ge 2$ and $k+\frac{\bp-\bm}{2}>1$ for $k\in \mathbb{N}$. 
Set $\varepsilon=(-1)^k$.
Then $I_{\varepsilon}(k+\bp-2)$ contains
a unique irreducible submodule 
$\Pi^{b^+,b^-}_{k,0}$.
It is unitarizable and its $\OO(V^+)\times\OO(V^-)$-types are given by
$$
\bigoplus_{\substack{m,n\in \mathbb{N}\\ m-n\in 2\mathbb{N}+k}}
\H^m(V^+)\boxtimes\H^n(V^-).
$$
\end{proposition}
Here we particularly focus on the component 
\begin{equation}\label{theminimalktype}
\H^{k,0}(V)=\H^k(V^+)\boxtimes\H^0(V^-)=\H^k(V^+)\boxtimes\mathbf{1}
\end{equation}
in the above $\OO(V^+)\times\OO(V^-)$-types, which we call 
{\it the minimal $\OO(V^+)\times\OO(V^-)$-type of $\Pi^{b^+,b^-}_{k,0}$}.
The observation (\ref{dpssections}) implies that 
any element in the minimal $\OO(V^+)\times\OO(V^-)$-type 
can be written as 
\begin{equation}\label{basisofminimalK}
f(g; h)=h((g^{-1}z)^+)\Vert (g^{-1}z)^+\Vert^{-2k-\bp+2},\quad 
g\in\mathrm{O}(V),
\end{equation}
with varying $h\in \H^k(V^+)$.

\section{Theta liftings}
\subsection{Generalized theta series}
Let $L$ be a nondegenerate even lattice of signature $(\bp,\bm)$
(signature of $L_\R=L\otimes\R$ understood throughout).
Put $\sig(L)=\bp-\bm$.
Denote by $L'$ the dual of $L$;
$L'=\{\lambda\in L_\Q=L\otimes\Q\mid (\lambda,\mu)\in \Z\hbox{ for all }\mu
\in L\}$.
We set $q(\lambda)=(\lambda,\lambda)/2$ for $\lambda\in L_\Q$
whose restriction to $L$ then takes values in $\Z$.
Moreover the modulo $1$ reduction of $q$ on $L'$ 
gives a $\Q/\Z$-valued quadratic form
on the finite discriminant group $D_L=L'/L$, whose associated
$\Q/\Z$-valued bilinear form coincides with
the modulo $1$ reduction of the bilinear form on $L'$.

The group $\OO(L)$ of automorphisms of $L$ gives
a discrete subgroup of the Lie group $\OO(L_\R)$.
Let $\mathrm{SO}^+(L_\R)$ be the connected component of $\OO(L_\R)$
and set $\mathrm{SO}^+(L)=\OO(L)\cap \mathrm{SO}^+(L_\R)$. 
Since $L'$ is stable under $\mathrm{SO}^+(L)$, we find that
$\mathrm{SO}^+(L)$ acts on $D_L$.
Then we define the discriminant kernel
\begin{equation}\label{discriminantkernel}
\varGamma_L=\{g\in \mathrm{SO}^+(L)\mid g\lambda-\lambda\in L
\hbox{ for all }\lambda\in L'\}.
\end{equation}

We fix an isometric isomorphism $v_0$ from $(L_\R, q)$ to the space 
$(V,Q)$ defined in (\ref{V}).
For every $x\in L_\R$, we write $v_0(x)=v_0^+(x)+v_0^-(x)$ 
according to the decomposition $V=V^+\oplus V^-$.
On the vector space $V$ we shall consider the Laplacian
$\Delta_b=\sum_{j=1}^b\partial^2/\partial x_j^2$ 
associated with the definite
form $\sum_{j=1}^b x_j^2$ on $\R^b$.
We note that it differs from
the Laplacian of $(V,Q)$ invariant under $\OO(V)$.
Denote by $\mathcal{P}^m(V)$ be the space of all homogeneous polynomials
on $V$ of degree $m$.
We say that $p\in\mathcal{P}^m(V)$ is homogeneous of degree $(m^+,m^-)$,
if $m=m^++m^-$ and $p$ is homogeneous of degree $m^+$ in
the first $\bp$ variables of $V^+$, and of degree $m^-$ in 
the last $\bm$ variables of $V^-$.
We denote by $\mathcal{P}^{m^+,m^-}(V)$ the subspace of 
all such polynomials.
It should be emphasized that 
$\H^{k,0}(V)=\H^k(V^+)\boxtimes\mathbf{1}$ defined in (\ref{theminimalktype})
is regarded as a subspace of $\mathcal{P}^{k,0}(V)$.

Now we take $p(x)\in \mathcal{P}^{k,0}(V)$.
Then for $\tau=x+iy\in\mathbb{H}_1$ ($=$ the complex upper half plane),
$g\in \mathrm{O}(L_\R)$,
and $\gamma\in D_L$, we define
\begin{align}\label{theta}
\theta_{L+\gamma}(\tau,g;p)=y^{\frac{\bm}{2}}\sum_{\lambda\in L+\gamma}
&\left[\mathrm{exp}\left(-\frac{\Delta_b}{8\pi y}\right)p\right]
\bigl(v_0(g^{-1}\lambda)\bigr)\\
&\times\be\left(\tau Q(v_0^+(g^{-1}\lambda))+\overline{\tau}
Q(v_0^-(g^{-1}\lambda))\right),  \notag
\end{align}
where $\be(x)=e^{2\pi ix}$ for $x\in \C$.
It will be useful to introduce a more general theta function defined by
 \begin{multline}\label{generaltheta}
 \theta_{L+\gamma}\left(\tau,g;
\begin{pmatrix} \alpha\\ \beta\end{pmatrix}; p\right)
= y^{\frac{\bm}{2}}\sum_{\lambda\in L+\gamma}
\left[\mathrm{exp}\left(-\frac{\Delta_b}{8\pi y}\right)p\right]
\bigl(v_0(g^{-1}(\lambda+\beta))\bigr)\\
\times \be\left(\tau Q\bigl(v_{0}^{+}(g^{-1}(\lambda+\beta))\bigr)+
\overline{\tau}Q\bigl(v_{0}^{-}(g^{-1}(\lambda+\beta))\bigr)-
\bigl(\lambda+\frac{\beta}{2},\alpha\bigr)\right)
 \end{multline}
with additional $\alpha$, $\beta\in L_\R$.

Denote by $\Mp_2(\R)$ the two-fold metaplectic 
cover of $\SL_2(\R)$. 
It can be written as the group of pairs $(A,\phi(\tau))$, 
where $A=\bigl(\begin{smallmatrix}a & b \\ 
                                  c & d \end{smallmatrix}\bigl)
\in\SL_2(\R)$ 
and $\phi: \mathbb{H}_1\to\C$ is a holomorphic function on 
$\mathbb{H}_1=\{z\in\C\mid \mathrm{Im}(z)>0\}$ 
satisfying $\phi(\tau)^2=c\tau +d$. 
The product rule is defined by
$(A,\phi_1(\tau))\cdot(B,\phi_2(\tau))=(AB,\phi_1(B\tau)\phi_2(\tau))$,
where $A\tau=(a\tau+b)(c\tau+d)^{-1}$ is the standard action of 
$\SL_2(\R)$ on $\mathbb{H}_1$.
We write $\Mp_2(\Z)$ for 
the inverse image of $\SL_2(\Z)$ under 
the covering map.
Given an even lattice $L$, we let $\mathfrak{e}_\gamma$ for
$\gamma\in D_L$ be the basis of $\C[D_L]$ so that 
$\mathfrak{e}_\gamma\mathfrak{e}_\delta=\mathfrak{e}_{\gamma+\delta}$.
We then put an inner product $\langle\cdot,\cdot\rangle_L$ 
on $\C[D_L]$ as 
\begin{equation}\label{innerweil}
 \Bigl\langle\sum_{\gamma\in D_L}a_{\gamma}\mathfrak{e}_\gamma,
\sum_{\delta\in D_L}b_\delta\mathfrak{e}_\delta\Bigr\rangle_L
=\sum_{\gamma\in D_L}a_\gamma\overline{b_\gamma}.
\end{equation}
Recall the Weil representation 
$\rho_L: \Mp_2(\Z)\to \hbox{Aut }\C[D_L]$ 
associated with $L$ defined by
$$
\rho_L(T)\mathfrak{e}_\gamma=\mathbf{e}(q(\gamma))\mathfrak{e}_\gamma,
\quad  
\rho_L(S)\mathfrak{e}_\gamma=
\frac{\mathbf{e}(-\sig(L)/8)}{\sqrt{|D_L|}}
\sum_{\delta\in D_L} \mathbf{e}(-(\gamma,\delta))\mathfrak{e}_\delta,
$$
where $S=\bigl(\left(\begin{smallmatrix} 0 & -1 \\ 
                          1 & 0\end{smallmatrix}\bigr), \sqrt{\tau} \right)$ 
and $T=\bigl(\left(\begin{smallmatrix} 1& 1\\ 0 &1
\end{smallmatrix}\right), 1\bigr)$,
the standard generators of $\Mp_2(\Z)$.
It is unitary with respect to the inner product (\ref{innerweil}).

Now we define the $\C[D_L]$-valued theta function 
\begin{equation}\label{weiltheta}
\Theta_L\left(\tau,g; \begin{pmatrix}\alpha\\ \beta			  
			 \end{pmatrix};p\right)=\sum_{\gamma\in D_L}
\theta_{L+\gamma}\left(\tau,g;\begin{pmatrix}\alpha\\ \beta
		       \end{pmatrix};p\right)\mathfrak{e}_{\gamma}.
\end{equation}
If $\alpha=\beta=0$, then we will write $\Theta_L\left(\tau,g;
\left(\begin{smallmatrix} 0 \\ 0       
      \end{smallmatrix}\right);p\right)$ simply as 
$\Theta_{L}(\tau,g;p)$.
Since $\varGamma_L$ acts on $D_L$ trivially, we see that
$\Theta_{L}(\tau,g;p)$ is left $\varGamma_{L}$-invariant.
On the other hand, \cite[Theorem 4.1]{Bo} shows that 
(\ref{weiltheta}) satisfies the following 
transformation formula:
\begin{proposition}\label{thetatransformation}
Let $k\in\mathbb{N}$
and suppose $p(x)\in\mathcal{P}^{k,0}(V)$.
Then
\begin{equation}
\Theta_L\left(M\tau,g; \begin{pmatrix} a\alpha+b\beta\\ c\alpha+d\beta
\end{pmatrix}; 
p\right)=\phi(\tau)^{2k+\sig(L)}\rho_L(M,\phi)
\Theta_L\left(\tau,g;\begin{pmatrix} \alpha \\ \beta
		 \end{pmatrix}; p\right)
\end{equation}  
for every $(M,\phi)\in\mathrm{Mp}_2(\mathbb{Z})$ with 
$M=\bigl(\begin{smallmatrix} a &b\\c &d \end{smallmatrix}\bigl)\in \SL_2(\Z)$.
\end{proposition}

\subsection{Lifts of vector-valued modular forms}
Let $\nu\in \frac12\Z$ and assume that $\nu\ge \frac32$. 
A $\C[D_L]$-valued holomorphic function $f$ on $\mathbb{H}_1$
is called a holomorphic modular form of weight $\nu$ and type
$\rho_L$ for the group $\Mp_2(\Z)$, if 
$f(M\tau)=\phi(\tau)^{2\nu}\rho_L(M,\phi)f(\tau)$
for every $(M,\phi)\in \Mp_2(\Z)$ and $f$ is holomorphic at the cusp $i\infty$.
Such a modular form has a Fourier expansion of the form
\begin{equation}\label{ellipticfourier}
f(\tau)=\sum_{\gamma\in D_L}f_{L+\gamma}(\tau)\mathfrak{e}_\gamma
=\sum_{\gamma\in D_L}\sum_{n\in q(\gamma)+\Z}
c(n,\gamma)\be(n\tau)\mathfrak{e}_\gamma,
\end{equation}  
where $c(n,\gamma)=0$ if $n<0$.
If furthermore $c(0,\gamma)=0$ for all $\gamma\in D_L$, 
then we say that $f(\tau)$ is cuspidal. 
We denote by $S_\nu(D_L)$ the space of holomorphic 
cusp forms of weight $\nu$ and type $\rho_L$.
If $f$, $g$ are modular forms of the same weight $\nu$ and type $\rho_L$,
and at least one of $f$ and $g$ is cuspidal,
then we can consider their Petersson inner product
\begin{equation}\label{petersson}
 (f,g)_\tau=\int_{\SL_2(\Z)\backslash \mathbb{H}_1}
\bigl\langle f(\tau),g(\tau)\bigr\rangle_L y^\nu\,\frac{dxdy}{y^2}.
\end{equation}

Now set $\nu=k+\frac{\sig(L)}{2}$ with 
$k\in\mathbb{N}$ and suppose that $\nu>\frac32$.
Then Proposition \ref{thetatransformation} allows us to define
\begin{equation}\label{thetalifting}
\vartheta_L(f; p)(g)=(f(\tau), \Theta_L(\tau,g; p))_\tau
\end{equation}
for every $f(\tau)\in S_\nu(D_L)$ and 
$p(x)\in\mathcal{P}^{k,0}(V)$. 
Here it should be noted that $\Theta_L(\tau,g; p)$ is not holomorphic 
in $\tau$ when $\bm> 0$, but the definition 
(\ref{petersson}) still has a meaning.
Consequently the pairing (\ref{thetalifting}) 
gives a left $\varGamma_L$-invariant function on $\OO(L_\R)$, and 
we will investigate it, a lifting of $f(\tau)$ to $\OO(L_\R)$,
precisely in the subsequent sections.

\section{Theta lifts of Poincar\'e series} 
In this section we will consider the liftings (\ref{thetalifting}) 
of Poincar\'e series belonging to $S_\nu(D_L)$ 
by exploiting the theory of harmonic polynomials,
\cite{Vi}.

\subsection{Representations of $\OO(n)$}\label{vilenkinharmonic}
The space $\mathcal{P}^{k,0}(V)$ is naturally identified 
with $\mathcal{P}^k(V^+)$ of homogeneous polynomials of degree 
$k$ on $V^+=\R^\bp$, on which $\OO(V^+)$ acts.
Then we recall the $\OO(n)$-module structure of $\mathcal{P}^k(\R^n)$
described by Vilenkin \cite[Chapter IX]{Vi}.

Consider the Euclidean space $\R^n$ with the standard orthonormal basis
$\{e_j\}_{j=1}^n$, where $e_j$ is the unit vector with $1$ 
on $j$-th entry, $0$ others.
We write $x=\sum x_ie_i\in\R^n$ 
as $x={}^t(x_1,\ldots,x_n)$ with $\|x\|^2
=\sum_{j=1}^nx_j^2$ positive definite. 
The Laplacian on $\R^n$ is given by 
$\Delta_n=\sum_{j=1}^n\partial^2/\partial x_j^2$.
It is invariant under the $\OO(n)$-action
$g\cdot f(x)=f(g^{-1}x)$, $f\in \mathcal{P}^k(\R^n)$.
Let $\H^k(\R^{n})$ be the subspace of harmonic polynomials in 
$\mathcal{P}^k(\R^n)$;
$$
\H^k(\R^n)=\{f\in \mathcal{P}^k(\R^n)\mid \Delta_n f=0\}.
$$
As is well known, $\H^{k}(\R^{n})$ is 
an irreducible $\OO(n)$-module 
and $\mathcal{P}^k(\R^n)$ is decomposed into the sum
\begin{equation}\label{decomposition-poly} 
\mathcal{P}^k(\R^n)=\bigoplus_{0\leq j \leq \left[\frac{k}{2}\right]}
\|x\|^{2j}\H^{k-2j}(\R^n).
\end{equation}
For every $p(x)\in \mathcal{P}^k(\R^n)$ let us set 
\begin{equation}\label{hpoperator}
Hp(x)=\sum_{j=0}^{\left[\frac{k}{2}\right]}
\frac{(-1)^j\Gamma\left(\frac{n}{2}+k-j-1\right)}
{4^jj!\Gamma\left(\frac{n}{2}+k-1\right)}\|x\|^{2j}
\Delta_n^jp(x).
\end{equation}
Then $Hp(x)\in \H^k(\R^n)$ and $H(Hp)=Hp$.
Therefore the linear map $H$ gives the orthogonal projection from 
$\mathcal{P}^k(\R^n)$ onto $\H^k(\R^n)$ in 
(\ref{decomposition-poly}) (see \cite[p.446, (15)]{Vi}). 

Next we consider the decomposition of $\H^k(\R^n)$ 
restricted to the subgroup $\OO(n-1)$ that stabilizes
the point ${}^t(1, 0,\ldots,0)\in\R^n$. 
For each fixed integer $l$, the polynomials $H(x_1^{k-l}h(x'))$ with varying
$h(x')\in\H^l(\R^{n-1})$, $x'={}^t(x_2, \ldots,x_n)\in \R^{n-1}$,
span a subspace of $\H^k(\R^n)$ that we denoted by $\H^{nkl}$.
According to \cite[Chapter IX, \S 3.5]{Vi}, we then obtain
\begin{equation}\label{decompositionharmonicspace}
\H^k(\R^n)=\bigoplus_{l=0}^k \H^{nkl},
\end{equation}
which is indeed the irreducible decomposition of $\H^k(\R^n)$
as an $\OO(n-1)$-module.
By definition (\ref{hpoperator}), we see that
\begin{equation}\label{onestepprojection}
H(x_1^{k-l}h(x'))=h(x')Hx_1^{k-l}
=\frac{(k-l)!}{2^{k-l}}\|x\|^{k-l}C_{k-l}^{\frac{n-2}{2}+l}
\left(\frac{x_1}{\|x\|}\right)h(x') 
\end{equation}
(see \cite[p. 465, (9)]{Vi}), 
where $C_n^\lambda(z)$ $(\lambda>0)$ is the Gegenbauer polynomial
$$
C_n^\lambda(z)=\sum_{j=0}^{\left[\frac{n}{2}\right]}\frac{(-1)^k
\Gamma(\lambda+n-j)}{j!(n-2j)!\Gamma(\lambda)}(2z)^{n-2j}.
 $$
This and (\ref{decompositionharmonicspace}) then imply that
any element of $\mathcal{H}^{k}(\R^{n})$ is written as a 
linear sum of products of $h\in \mathcal{H}^{k-l}(\R^{n-1})$ 
and the Gegenbauer polynomials. 
By induction on $n$, this yields 
\begin{proposition}\label{harmonicbasis}{\rm (\cite[p.466, (2)]{Vi})}
Suppose $n\ge 2$ and take the standard ordered orthonormal basis 
$\{e_j\}_{j=1}^n$ of $\R^n$ as above.
Then $\H^k(\R^n)$ has a basis $\{h_\kappa^n(x)\}_\kappa$
consisting of the polynomials 
\begin{equation*}
h_\kappa^n(x)=(x_{n-1}\pm i x_n)^{k_{n-2}}\prod_{j=0}^{n-3}
\Vert x\Vert_j^{k_j-k_{j+1}}C_{k_j-k_{j+1}}^{k_{j+1}+\frac{n-j-2}{2}}
\left(\frac{x_{j+1}}{\Vert x\Vert_j}\right),
\end{equation*}
where $\Vert x\Vert^2_j=x_{j+1}^2+\cdots+x_n^2$ and 
$\kappa=(k_0,k_1,\ldots,k_{n-3}, \pm k_{n-2})$ varies over all 
multi-indices in $\Z^{n-1}$ such that $k=k_0\ge k_1\ge \cdots\ge k_{n-2}\ge 0$. 
If $n=2$, we understand that $\kappa=\pm k$ with an integer $k\ge 0$, and
$h^2_{\pm k}(x)=(x_1\pm ix_2)^k$.
\end{proposition}
\begin{remark}\label{remarkonbasis}
We note that the operator $H$ is written as a polynomial of $\Delta_n$.
It implies that the characterization of $h_\kappa^n(x)$ using 
the coordinates $x_j$ in Proposition \ref{harmonicbasis} naturally transforms 
under the action of $\OO(n)$ on ordered orthonormal bases of $\R^n$.
Therefore we may keep the same symbol $h_\kappa^n(x)$, $x\in\R^n$, to 
denote the bases of $\H^k(\R^n)$ regardless of
the particular choice of ordered orthonormal basis of $\R^n$.
This convention must be kept in mind when a choice of an
ordered orthonormal basis is strictly considered.
\end{remark}

\subsection{Theta lifts of Poincar\'e series}
Let $L$ be an even lattice of signature $(\bp,\bm)$ as before.
For a $\C[D_L]$-valued function $f(\tau)$ on $\mathbb{H}_1$,
we define an action of $(M,\phi)\in \Mp_2(\Z)$ by
$$
f\vert_{\nu,L}(M,\phi)(\tau)=\phi(\tau)^{-2\nu}\rho_L(M,\phi)^{-1}
f(M\tau), \quad \nu=k+\frac{\sig(L)}{2}.
$$
Suppose that $\nu\ge \frac32$.
For every $\beta\in D_L$, $0<m\in Q(\beta)+\Z$, and $\tau\in \mathbb{H}_1$,
we put $\mathfrak{e}_\beta(m\tau)=\be(m\tau)\mathfrak{e}_\beta$.
Then we define the $\C[D_L]$-valued 
Poincar\'e series $P_{m,\beta}(\tau)$ of weight 
$\nu$ and index $(m,\beta)$ by
\begin{equation}\label{poincareseries}
P_{m,\beta}(\tau)
=\frac{1}{2}\sum_{(M,\phi)\in 
\widetilde{\varGamma}_{\infty}\backslash\Mp_2(\Z)}
\mathfrak{e}_\beta(m\tau)\vert_{\nu,L}(M,\phi),
\end{equation}
where $\widetilde{\varGamma}_{\infty}=\langle T\rangle$.
It belongs to $S_\kappa(D_L)$ and satisfies
$$
(f(\tau),P_{m,\beta}(\tau))_\tau=\frac{2\Gamma(\nu-1)}{(4\pi m)^{\nu-1}}
c(m,\beta)
$$
for every cusp form $f(\tau)\in S_\nu(D_L)$ having the Fourier expansion
(\ref{ellipticfourier}) 
(see \cite[Theorem 1.4 and Proposition 1.5]{Br}).

Now we describe the theta liftings (\ref{thetalifting}) of $P_{m,\beta}(\tau)$.
\begin{theorem}\label{harmonicprojection}
Suppose that $p(x)\in \mathcal{P}^{k,0}(V)$. 
Then
\begin{equation}\label{latticepoincare}
 \vartheta_L(P_{m,\beta}; p)(g)=\frac{2\Gamma\bigl(k+\frac{\bp}{2}-1\bigr)}
{(2\pi)^{k+\frac{\bp}{2}-1}}
\sum_{\substack{\lambda\in L+\beta\\ Q(\lambda)=m}}
\frac{Hp\bigl(v_0^+(g^{-1}\lambda)\bigr)}{\|v_0^+(g^{-1}\lambda)\|^{2k+\bp-2}},
\end{equation}
where $Hp$ is the projection of $p$
defined in (\ref{hpoperator}).
In particular, one has
$$
\vartheta_L(P_{m,\beta}; p)(g)=\vartheta_L(P_{m,\beta}; Hp)(g).
$$
\end{theorem}
\begin{proof}
Definition (\ref{poincareseries}) and
Proposition \ref{thetatransformation} 
imply that $\vartheta_L(P_{m,\beta}; p)(g)$ is equal to
\begin{align*}
&\frac{1}{2}\int_{\SL_2(\Z)\backslash\mathbb{H}_1}
\sum_{(M,\phi)\in\tilde{\varGamma}_{\infty}\backslash\Mp_2(\Z)}
\bigl\langle e_\beta(mM\tau),\rho_L(M,\phi)\Theta_L(\tau,g;p)\bigr\rangle_L
 \phi(\tau)^{-2\nu}y^\nu\,\frac{dxdy}{y^2}\\
& =\int_{\SL_2(\Z)\backslash \mathbb{H}_1}
 \sum_{M\in \varGamma_\infty\backslash \SL_2(\Z)}
 \bigl\langle e_\beta(mM\tau),\Theta_L(M\tau,g;p)\bigr\rangle_L
 \mathrm{Im}(M\tau)^\nu\,\frac{dxdy}{y^2}
\end{align*}
with $\varGamma_\infty=\bigl\langle (\begin{smallmatrix}
 1 & 1\\
 0 & 1
\end{smallmatrix}\bigr)\bigr\rangle\subset \SL_2(\Z)$.
By the unfolding argument, we can compute it as
\begin{align*}
&2\int_0^\infty\int_0^1\mathbf{e}(m\tau)
\overline{\theta_{L+\beta}(\tau,g;p)}y^\nu\frac{dydx}{y^{2}}\\
=~&2\sum_{\substack{\lambda\in L+\beta\\ Q(\lambda)=m}}
\int_0^\infty
\left[\mathrm{exp}\left(-\frac{\Delta_{\bp+\bm}}{8\pi y}\right)p\right](v_0(g^{-1}\lambda))\mathrm{exp}\bigl(-4\pi Q(v_0^+(g^{-1}\lambda)) y\bigr)
y^{k+\frac{\bp}{2}-1}\frac{dy}{y}.
\end{align*}
Now putting the right-hand side of 
$$
\left[\mathrm{exp}\left(-\frac{\Delta_{\bp+\bm}}{8\pi y}\right)p\right]
(v_0(g^{-1}\lambda))
=\sum_{j=0}^{\left[\frac{k}{2}\right]}
\frac{(-1)^{j}[\Delta_\bp^jp](v_0^+(g^{-1}\lambda))}{j!(8\pi y)^{j}}
$$
into the final line above and carrying out the integrals over $y$, 
we then obtain (\ref{latticepoincare}).
Since $H(Hp)=Hp$, the identity (\ref{latticepoincare}) concludes 
the second assertion.
\end{proof}
Since the family of Poincar\'e series (\ref{poincareseries})
span the space $S_\kappa(D_L)$,
Theorem \ref{harmonicprojection} gives the following important identities:
\begin{corollary}\label{hpidentity} 
For every $p(x)\in\mathcal{P}^{k,0}(V)$ and $f(\tau)\in S_\nu(D_L)$, one has
$$
\vartheta_L(f; p)(g)=\vartheta_L(f; Hp)(g).
$$
\end{corollary}
\subsection{Theta lifts of vector-valued modular forms}
We have a natural equivalence $\mathcal{P}^{k,0}(V)\simeq 
\mathcal{P}^k(V^+)\boxtimes\mathbf{1}$ as  
$\OO(V^+)\times\OO(V^-)$-modules.
Then $\H^{k,0}(V)=\H^k(V^+)\boxtimes\mathbf{1}$ is identified 
with $\H^k(V^+)$ as $\OO(V^+)$-modules.
The fixed isometry $v_0$ determines
a maximal compact subgroup $K_{v_0}$ of $\OO(L_\R)$
that preserves the pullback of the decomposition $V=V^+\oplus V^-$
to $L_\R$.
Thus $K_{v_0}$ is isomorphic to $\OO(V^+)\times\OO(V^-)$ and 
acts on $H_k=\H^{k,0}(V)$. 
Denote by $\tau_k$ this $K_{v_0}$-representation on $H_k$ and 
let $\tau_k^\vee$ be its dual representation on 
$H_k^\vee=\hbox{Hom}_\C(H_k,\C)$.
For any basis $A=\{h\}$ of $H_k$, we take its dual basis $\{h^\vee\}$ of 
$H_k^\vee$.
Then we define a $H_k^\vee$-valued function on $\mathrm{O}(L_\R)$ by
\begin{equation}\label{liftgeneral}
\vartheta_L(f)(g)=\sum_{h\in A}\vartheta_L(f; h)(g)
h^\vee
\end{equation}
for every $f\in S_\nu(D_L)$,
which satisfies $\vartheta_L(f)(g)(h)=\vartheta_L(f; h)(g)$ 
for every $h\in H_k$.
Clearly it does not depend on any choice of $A$.
The following properties are direct consequences of the left 
$\varGamma_L$-invariance of theta series and the definitions 
(\ref{thetalifting}) and (\ref{liftgeneral}):
\begin{lemma}\label{automorphicproperty}
For every $h\in H_k$,
$\gamma\in \varGamma_L$, and $w\in K_{v_0}$ one has
$$
\vartheta_L(f; h)(\gamma g w)=\vartheta_L(f; \tau_k(w)h)(g)\hbox{ and }
\vartheta_L(f)(\gamma g w^{-1})=\tau_k^\vee(w)\vartheta_L(f)(g).
$$
\end{lemma}
\section{Fourier expansions of theta liftings}
Let $L$ be an even lattice of signature $(\bp,\bm)$.
In this section, we will compute the Fourier expansions of the lifting
$\vartheta_L(f)(g)$ defined in (\ref{liftgeneral}).
Borcherds \cite{Bo} established a general method to compute the expansions of 
theta liftings by decomposing $\Theta_L$ into a combination of 
generalized theta series associated to a smaller lattice.
Our idea is to combine his arguments with our observation
obtained in Corollary \ref{hpidentity}.
\subsection{A smaller lattice and a maximal $\Q$-parabolic subgroup}
\label{smallerlattice}
Firstly let us recall the arguments in \cite{Bo} and \cite{Br}.
Take and fix any primitive isotropic vector $z\in L$
and an element $z'\in L'$ with $(z,z')=1$.
Then we can take the lattice
$$
L_1=L\cap \langle z\rangle^\perp\cap\langle z'\rangle^\perp
$$
and its dual $L_1'\subset (L_1)_\Q$. 
Over $\Q$ we have the orthogonal decomposition 
\begin{equation}\label{rationalorthogonal}
L_\Q=(L_1)_\Q\oplus\langle z, z'\rangle_\Q,
\end{equation}
and hence $L_1$ is an even lattice of signature $(\bp-1,\bm-1)$.
For each $x\in L_\R$ let $x_{L_1}$ be the orthogonal projection 
of $x$ to the subspace $(L_1)_\R$, which is explicitly given by
$$
x_{L_1}=x-(x,z)z'-\bigl((x,z')-(x,z)(z',z')\bigr)z.
$$

Denote by $P_z$ be the maximal $\Q$-parabolic subgroup of $\OO(L_\R)$ 
stabilizing the isotropic line $\langle z\rangle$.
It admits the Levi decomposition $P_z=M_zN_z$ with the abelian unipotent
radical $N_z$ and the Levi subgroup $M_z= P_z\cap P_{z'-q(z')z}$, where
$P_{z'-q(z')z}$ is the maximal $\Q$-parabolic subgroup stabilizing 
the isotropic line $\langle z'-q(z')z\rangle$.
We have group isomorphisms $n_z: (L_1)_\R\simeq N_z$ and 
$m_z: \GL_1(\R)\times \OO((L_1)_\R)\simeq  M_z$,
which are characterized by the following identities:
\begin{align}
&n_z(u)z=z,\quad n_z(u)x=-(u,x)z+x,\quad
n_z(u)z'=-q(u)z+u+z', \label{bijectionlevi1}\\
&m_z(a,g_1)z=az,\quad m_z(a,g_1)x=g_1x,\quad
m_z(a,g_1)z'=a^{-1}z'+(a-a^{-1})q(z')z \label{bijectionlevi2}
\end{align}
for every $u, x\in (L_1)_\R$, $a\in \R^\times$, and $g_1\in \OO((L_1)_\R)$.
We note that $P_z$ stabilizes the sequence $\Q z\subset 
(L_1)_\Q\oplus\Q z\subset L_\Q$.

Recall the fixed isometry $v_0: (L_\R, q)\simeq (V, Q)$.
Denote by $L_\R^\pm$ the inverse images of $V^\pm$ under $v_0$.
Once again, we write every $x\in L_\R$ as $x=x^++x^-$ according to 
the orthogonal decomposition 
$L_\R=L_\R^+\oplus L_\R^-$.
For the fixed $z\in L$ we put 
\begin{equation}\label{realhyperbolic}
z^\ast=\frac{z^+-z^-}{2\Vert z^+\Vert^2}\in L_\R,
\end{equation}
then $q(z^\ast)=0$ and $(z,z^\ast)=1$.
Hence $\langle z,z^\ast\rangle=\langle z^+\rangle
\oplus\langle z^-\rangle$ is a hyperbolic plane in $L_\R$, and
provides the other orthogonal decomposition of $L_\R$
than (\ref{rationalorthogonal}):
\begin{equation}\label{grassmannorthogonal}
L_\R=W_\R \oplus\langle z,z^\ast\rangle,\quad
W_\R=\langle z^+\rangle^\perp\cap\langle z^-\rangle^\perp.
\end{equation}
We denote by $x_W$ the orthogonal projection
of $x\in L_\R$ to $W_\R$. 
Of course the hyperbolic planes $\langle z,z'\rangle$ and
$\langle z,z^\ast\rangle$ do not coincide in general, but
the map $x\mapsto x_W$ induces an isometric isomorphism
$(L_1)_\R\simeq W_\R$ by the following reasons:
if $x\in\langle z\rangle^\perp= W_\R\oplus \langle z\rangle$,
then we have  $x_W=x-(x,z^\ast)z$ and $\Vert x\Vert=\Vert x_W\Vert$,
and by definition we have $(L_1)_\R\subset \langle z\rangle^\perp$ and 
$(L_1)_\R\cap\langle z\rangle=\{0\}$.

We also define a linear map $v_1$ from $L_\R$ into $V$ as 
\begin{equation}\label{v1}
v_1(x)=v_0(x_W), \quad x\in L_\R.
\end{equation} 
In particular, $v_1$ vanishes on 
$\langle z^+\rangle\oplus\langle z^-\rangle
=\langle z,z^\ast\rangle$, 
and its restriction to $(L_1)_\R$
gives an isometry of $(L_1)_\R$ onto its image 
$v_1((L_1)_\R)=v_0(W_\R)=v_1(L_\R)$.

\begin{lemma}\label{parabolicreduction}
Let $g_z=n_z(u)m_z(a,g_1)\in P_z$.
Then one has that $v_1(g_z^{-1}x)=v_1(g_1^{-1}x_{L_1})$ for 
every $x\in (L_1)_\R\oplus\langle z\rangle$, and 
also that $(g_z^{-1}z^*)_{L_1}=g_1^{-1}(z^*)_{L_1}-u$.
\end{lemma}
\begin{proof}
Since $v_1$ vanishes on $\langle z\rangle$, the first
assertion follows from (\ref{bijectionlevi1}) and (\ref{bijectionlevi2}). 
The condition that $(z,z^*)=1$ implies that $z^*=cz+z'+(z^*)_{L_1}$ with $c\in \R$. 
Then (\ref{bijectionlevi1}) and (\ref{bijectionlevi2}) conclude 
the second assertion.
\end{proof}

\subsection{Reduction of theta series to smaller lattices}
In this section we suppose that $\bp>2$. 
Keep $z\in L$, $z'\in L'$, and $P_z$ as defined earlier.
As to the map $v_1$ in (\ref{v1}), we write $v_1(x)=v_1^+(x)+v_1^-(x)$ 
for every $x\in L_\R$ according to the decomposition $V=V^+\oplus V^-$.
Take $0\neq  v_0^+(z)/\Vert v_0^+(z)\Vert\in V^+$, and
extend it to an ordered orthonormal basis of $V^+$ by taking
appropriate elements of $v_1^+(L_\R)$.
This new basis differs from the original one $\{e_1,\ldots, e_\bp\}$ 
only by the $\OO(V^+)$-action, and we treat each $h_\kappa^\bp(x)\in 
\H^{k,0}(V)$ 
using this new basis after Remark \ref{vilenkinharmonic}.

With reference to (\ref{decompositionharmonicspace}) and Proposition 
\ref{harmonicbasis}, we define polynomials
\begin{equation}\label{testpolynomial}
p_\kappa(v_0(x))=\left(v_0(x), 
\frac{v_0^+(z)}{\Vert v_0^+(z)\Vert}\right)^{k_0-k_1}
h_{\kappa^{(1)}}^{\bp-1}(v_1^+(x))
\end{equation}
for $x\in L_\R$, $\kappa=(k_0,k_1,\ldots, \pm k_{\bp-2})$ with
$k=k_0\ge k_1\ge\cdots\ge k_{\bp-2}\ge 0$,
where we set $\kappa^{(1)}=(k_1,\ldots, \pm k_{\bp-2})$ in the right-hand side.
By definition we see that $p_\kappa(v_0(x))=p_\kappa(v_0^+(x))$, and hence
$p_\kappa\in \mathcal{P}^{k,0}(V)$.
By (\ref{onestepprojection}),
we note that 
\begin{equation}\label{auxiliarypolynomial}
Hp_\kappa(v_0(x))=\frac{(k_0-k_1)!}
{2^{k_0-k_1}} h_\kappa^\bp(v_0(x)).
\end{equation}
Thus $\{Hp_\kappa\}_\kappa$ gives a basis of $\H^{k,0}(V)$.
Combined with Corollary \ref{hpidentity},
the polynomials $p_\kappa$ play an auxiliary
role in our computation of 
the Fourier expansion of $\vartheta_L(f; h_\kappa^\bp)(g)$.

Let $N$ be the positive integer uniquely determined by
$(L,z)=N\mathbb{Z}$ and $\zeta\in L$ be a vector satisfying $(\zeta,z)=N$. 
Then $L$ can be written as 
\begin{equation}\label{latticedecomposition}
L=L_1\oplus \Z\zeta\oplus \Z z
\end{equation}
as an additive group (see \cite[Proposition 2.2]{Br}).

\begin{lemma}\label{basicdecomposition} 
Suppose $\bp>2$.
Let $z\in L$ be a primitive isotropic vector, and
let $p_\kappa$ be as in (\ref{testpolynomial}).
For every $\gamma\in D_L$ and $g_z=n_z(u)m_z(a,g_1)\in  P_z$, one has
\begin{align*}
\theta_{L+\gamma}(\tau,g_z;p_\kappa)&
=\frac{y^{\frac{\bm-1}{2}}a^{k_0-k_1}|a|}
{\sqrt{2}\Vert v_0^+(z)\Vert^{k_0-k_1+1}}
\sum_{\lambda\in L_1\oplus\Z\zeta+\gamma}
\sum_{n\in \Z}\left(\frac{(\lambda,z)\overline{\tau}+n}{-2iy}\right)^{k_0-k_1}
h_{\kappa^{(1)}}^{\bp-1}(v_1^+(g_z^{-1}\lambda))  \\
\times&\mathbf{e}\left(\tau Q(v_1^+(g_z^{-1}\lambda))+
\overline{\tau}Q(v_1^-(g_z^{-1}\lambda))
-n(\lambda, ag_zz^\ast)
-\frac{a^2|(\lambda,z)\tau+n|^2}{4iy\Vert v_0^+(z)\Vert^2}
\right).
\end{align*}
\end{lemma}
\begin{proof}
Let us write $v^{\pm}_{0}(g_z^{-1}u)$ simply as $v^\pm(u)$ and also set 
\begin{align*}
f(\lambda,v; n)=y^{\frac{\bm}{2}}\left[{\rm exp}
\left(-\frac{\Delta}{8\pi y}\right)p_\kappa\right]
(v^{+}(\lambda+nz)) \mathbf{e}\left(
\tau Q(v^{+}(\lambda+nz))+\overline{\tau}Q(v^{-}(\lambda+nz))\right).
\end{align*}
The Poisson summation formula
and (\ref{latticedecomposition}) imply that
$\theta_{L+\gamma}(\tau,g;p_\kappa)$ is equal to
$$
\sum_{\lambda\in L_1\oplus \Z \zeta+\gamma}
\sum_{n\in \Z}f(\lambda,v; n)=
\sum_{\lambda\in L_1\oplus \Z \zeta+\gamma}
\sum_{n\in \Z}\widehat{f}(\lambda,v; n),
$$
where $\widehat{f}(\lambda,v; n)$ is the partial Fourier transform 
of $f(\lambda,v; n)$ in the variable $n$.
We refer to \cite[p.509]{Bo} for
the explicit description of the factor $[\mathrm{exp}(-\Delta/8\pi y)p_\kappa]
(v^+(\lambda+nz))$ occurring in $f(\lambda,v; n)$.
Then we follow the computation of 
$\widehat{f}(\lambda, v;n)$ given in the proof of \cite[Lemma 5.1]{Bo}
with noting that $g_z^{-1}z=a^{-1}z$, and hence
$h^{\bp-1}_{\kappa^{(1)}}(v_1^+(g_z^{-1}(\lambda+nz)))=
h^{\bp-1}_{\kappa^{(1)}}(v_1^+(g_z^{-1}(\lambda)))$.
Since $\Delta h_{\kappa^{(1)}}^{\bp-1}=0$, 
we conclude the required equality.
\end{proof}
Take a sublattice
$M'=\{\lambda\in L'\mid (\lambda,z)\equiv 0\ \mathrm{mod}\ N\}$ of $L'$. 
Then there is a projection map $\pi: M'\to L_1'$ defined  by
\begin{equation}\label{prM'}
\pi(\lambda)=\lambda_{L_1}-\frac{(\lambda,z)}{N}\zeta_{L_1},\quad 
\lambda\in M'
\end{equation}
(see \cite[p.41, (2.7)]{Br}).
It satisfies that $\pi(L)=L_1$, and thus
induces a surjective map $M'/L\to L_1'/L_1=D_{L_1}$, 
which we also denote by $\pi$.

For every $g_z=n_z(u)m_z(a,g_1)\in  P_z$, we put
\begin{equation}\label{mu}
\mu(g_z)=-z'+ag_zz^\ast=-z'+z^\ast +u-q(u)z.
\end{equation}
Then $\mu(g_z)\in \langle z\rangle^\perp=W_\R\oplus\langle z\rangle
=(L_1)_\R\oplus\langle z\rangle$.
Now we recall the reduction formula of $\theta_{L+\gamma}$
(see \cite[Theorem 5.2]{Bo} and \cite[Theorem 2.4]{Br}):
\begin{proposition}\label{thetareduction}
Suppose $\bp>2$.
Let $\gamma\in L'$ and $g_z=n_z(u)m_z(a,g_1)\in  P_z$. 
For every $p_\kappa\in \mathcal{P}^{k,0}(V)$ defined in (\ref{testpolynomial}),
$\theta_{L+\gamma}(\tau, g_z; p_\kappa)$ is equal to
\begin{align*}
&\frac{a^{k_0-k_1}|a|}{\sqrt{2}\Vert v_0^+(z)\Vert^{k_0-k_1+1}}
\sum_{\substack{c,d\in \Z\\ c\equiv (\gamma,z)~\mathrm{mod}~N}}
\left(\frac{c\overline{\tau}+d}{-2iy}\right)^{k_0-k_1}
\mathbf{e}\left(\frac{-a^2|c\tau+d|^2}{4iy\Vert v_0^+(z)\Vert^2}
-d(\gamma,z')+\frac{cd(z',z')}{2}\right)\\
&\qquad\qquad \times\theta_{L_1+\pi(\gamma-cz')}
\left (\tau, g_1; \begin{pmatrix} d\mu(g_z)_{L_1}\\-c\mu(g_z)_{L_1}
	      \end{pmatrix}; h_{\kappa^{(1)}}^{\bp-1}\right).
\end{align*}
\end{proposition}
\begin{proof}
Write $\gamma\in L'$ as $\gamma=\gamma_{L_1}+a_\gamma z'+b_\gamma z
\in (L_1)_\R\oplus\langle z', z\rangle$ with $\gamma_{L_1}\in L_1'$, 
$a_\gamma=(\gamma,z)\in \Z$, and $b_\gamma=(\gamma,z')-(\gamma,z)(z',z')
\in\Q$.
In particular, every 
$\lambda\in L_1\oplus \Z\zeta+\gamma$ can be uniquely written as
$$
\lambda=\lambda_{L_1}+cz'+f(c,\gamma)z
$$
with $c\in \Z$, $c\equiv a_\gamma$ mod $N$, 
$\lambda_{L_1}\in L_1+\gamma_{L_1}+(c-a_\gamma)\zeta_{L_1}/N$, and
$f(c,\gamma)=b_{\gamma}+\frac{(c-a_{\gamma})b_\zeta}{N}$ 
where $b_\zeta=(\zeta,z')-(\zeta,z)(z',z')\in \Q$ so that 
$\zeta=\zeta_{L_1}+Nz'+b_\zeta$.
Then we observe that $\gamma-cz'\in M'$ and
$\pi(\gamma-cz')=\gamma_{L_1}+(c-a_\gamma)\zeta_{L_1}/N$, and hence 
that $\lambda_{L_1}\in L_1+\pi(\gamma-cz')$.

By Lemma \ref{basicdecomposition},
$\theta_{L+\gamma}(\tau,g_z;p_\nu)$ is written as
\begin{align*}
&\frac{y^{\frac{\bm-1}{2}}a^{k_0-k_1}|a|}{\sqrt{2}
\Vert v_0^+(z)\Vert^{k_0-k_1+1}}
\sum_{\substack{c,d\in \Z\\ c\equiv (\gamma,z)\ (N)}}
\sum_{\lambda_{L_1}\in L_1+\pi(\gamma-cz')}
\left(\frac{c\overline{\tau}+d}{-2iy}\right)^{k_0-k_1}
h_{\nu_1}^{\bp-1}\left(v_1^+\bigl(g_z^{-1}(\lambda_{L_1}+cz')\bigr)\right)\\
&\quad \times \be\left(\tau Q\left(v_1^+
\bigl(g_z^{-1}(\lambda_{L_1}+cz')\bigr)\right)+
\overline{\tau}Q\left(v_1^-
\bigl(g_z^{-1}(\lambda_{L_1}+cz')\bigr)\right)\right)\\
&\qquad \times\be\left(-d\bigl(\lambda_{L_1}+cz'+f(c,\gamma)z, ag_zz^\ast\bigr)
-\frac{a^2|c\tau+d|^2}
{4iy\Vert v_0^+(z)\Vert^2}\right),
\end{align*}
where we have used $v_1^{\pm}(z)=0$, $(z,z)=0$, and 
$(\lambda_{L_1}+cz',z)=c$.
Lemma \ref{parabolicreduction} and the remark that
$\mu(g_z)\in (L_1)_\R\oplus\langle z\rangle$ imply that
\begin{align*}
v_1^\pm(g_z^{-1}(\lambda_{L_1}+cz'))&
=v_1^\pm(g_z^{-1}(\lambda_{L_1}-c\mu(g_z))+acz^\ast)\\
&=v_1^{\pm}(g_1^{-1}(\lambda_{L_1}-c\mu(g_z)_{L_1})).
\end{align*}

Regarding definition (\ref{generaltheta}), it remains therefore 
to prove the identities
\begin{align}\label{modeq}
&-d(\lambda_{L_1}+cz'+f(c,\gamma)z, ag_zz^\ast)\\
\equiv &-\left(\lambda_{L_1}-\frac{c\mu(g_z)_{L_1}}{2},d\mu(g_z)_{L_1}\right)
-d(\gamma,z')+\frac{(z',z')cd}{2}\ \hbox{ mod }1.\notag
\end{align}
Since $(z,ag_zz^\ast)=1$, we find that
$$
-d(f(c,\gamma)z, ag_zz^\ast)\equiv -d(\gamma,z')+cd(z',z')\ \hbox{ mod }1
$$
as in \cite[p.44]{Br}.
On the other hand, the identity
$$
-d(\lambda_{L_1}+cz', ag_zz^\ast)
=-\left(\lambda_{L_1}-\frac{c}{2}\mu(g_z), d\mu(g_z)
\right)-\frac{cd}{2}(z',z')
$$
is obtained from the calculation in \cite[pp.44-45]{Br},
where the first term in the right-hand side depends only on $\mu(g_z)_{L_1}$.
These conclude (\ref{modeq}), and the proof is completed.
\end{proof}

\subsection{Fourier expansions of theta lifts}
For $f(\tau)=\sum_{\gamma\in D_L}f_{L+\gamma}(\tau)
\frak{e}_\gamma\in S_\nu(D_L)$ and $r$, $t\in\Z$,
we define a $\C[D_{L_1}]$-valued function
$$
f_{L_1}\left(\tau; \begin{pmatrix} r\\ t \end{pmatrix}
\right)=\sum_{\lambda\in D_{L_1}} 
f_{L_1+\lambda}\left(\tau; \begin{pmatrix} r\\ t \end{pmatrix}\right)
$$
by putting
$$
f_{L_1+\lambda}\left(\tau; \begin{pmatrix} r \\ t			 
			\end{pmatrix}\right)
=\sum_{\substack{\delta\in M'/L\\
\pi(\delta)=\lambda}}\be\left(-r(\delta,z')-rt\frac{(z',z')}{2}\right)
f_{L+\delta+t z'}(\tau).
$$
Then it satisfies the relations
\begin{equation}\label{Bo5.3}
f_{L_1}\left(M\tau;  M\begin{pmatrix} r \\ t		  
		 \end{pmatrix}\right)
=\phi(\tau)^{2\nu}\rho_{L_1}(M,\phi)
f_{L_1}\left(\tau; \begin{pmatrix} r\\ t \end{pmatrix}\right)
\end{equation}
for all $(M,\phi)\in\Mp_2(\Z)$ (see \cite[Theorem 5.3]
{Bo} and \cite[Theorem 2.6]{Br}).
We write $f_{L_1}(\tau;0,0)$ simply as $f_{L_1}(\tau)$,
then  $f_{L_1}(\tau)\in S_\nu(D_{L_1})$.

Recall the basis $\{h_\kappa^\bp\}_\kappa$ of $\H^{k,0}(V)$ in
Proposition \ref{harmonicbasis}.
\begin{theorem}\label{fourierexpansion>2}
Suppose $\bp>2$.
Take $f(\tau)\in S_\nu(D_L)$, $\nu=k+\frac{\sig(L)}{2}$,
with the Fourier expansion (\ref{ellipticfourier}).
Then the value $\vartheta_L(f; h_\kappa^\bp)(g_z)$ at 
$g_z=n_z(u)m_z(a,g_1)\in P_z$ is equal to
\begin{align*}
&\frac{\delta_{kk_1}|a|}{\sqrt{2}\Vert v_0^+(z)\Vert}
\bigl(f_{L_1}(\tau),
\Theta_{L_1}(\tau, g_1; h^{\bp-1}_{\kappa^{(1)}})\bigr)_\tau\\
&\quad +\frac{(-i)^{k-k_1}(a/|a|)^{k-k_1}|a|^{k+\frac{\bp-1}{2}}}
{2^{k_1+\frac{\bp}{2}-3}(k-k_1)!
\Vert v_0^+(z)\Vert^{k+\frac{\bp-1}{2}}}
\sum_{\substack{\lambda\in L_1'\\ q(\lambda)>0}}
\biggl(\sum_{n|\lambda}n^{k+\frac{\bp-3}{2}}
\sum_{\substack{\delta\in M'/L\\ \pi(\delta)= \lambda/n}}
\be\bigl((n\delta, z')\bigr)
c\Bigl(\frac{q(\lambda)}{n^2},\delta\Bigr)\biggr)\\
&\qquad\quad\times \be\bigl((\lambda, u+ag_1 (z^\ast)_{L_1})\bigr)
\frac{h_{\kappa^{(1)}}^{\bp-1}(v_1^+(g_1^{-1}\lambda))}
{\Vert v_1^+(g_1^{-1}\lambda)\Vert^{k_1+\frac{\bp-3}{2}}}
K_{k_1+\frac{\bp-3}{2}}
\left(\frac{2\pi|a|\Vert v_1^+(g_1^{-1}\lambda)\Vert}
{\Vert v_0^+(z)\Vert}\right),
\end{align*}
where the sum $\sum_{n|\lambda}$ runs over all $n\in \mathbb{N}$ with
$\lambda/n\in L_1'$, and $K_\alpha(z)$ denotes the standard $K$-Bessel function.
\end{theorem}
\begin{proof}
According to Corollary \ref{hpidentity} and (\ref{auxiliarypolynomial}),
it suffices to evaluate 
$\vartheta_L(f)(g_z; p_\kappa)=(f(\tau), \Theta_L(\tau, g_z;p_\kappa))_\tau$
for $p_\kappa\in \mathcal{P}^{k,0}(V)$ in (\ref{testpolynomial}).
During this proof we write $\mu(g_z)$ simply as $\mu$.
By (\ref{mu}) we then see that $\mu_{L_1}=u+ag_1(z^\ast)_{L_1}$ 
(thus we have $\mu_{L_1}=u$ in the special case where
$\langle z,z^\ast\rangle=\langle z,z'\rangle$).
By Proposition \ref{thetareduction} we can write
$\vartheta_L(f)(g_z; p_\kappa)$ as
\begin{align*}
&\frac{a^{k_0-k_1}|a|}{\sqrt{2}\Vert v_0^+(z)\Vert^{k_0-k_1+1}}
\int_{\SL_2(\Z)\backslash \mathbb{H}_1}
\sum_{c,d\in \mathbb{Z}}
\be\left(\frac{-a^2|c\tau+d|}{4iy\Vert v_0^+(z)\Vert^2}\right)
\sum_{\substack{\gamma\in D_L\\ (\gamma,z)\equiv c~ (N)}}
f_{L+\gamma}(\tau)  \\
&\times\be\left(d(\gamma,z')-\frac{cd(z',z')}{2}\right)
\left(\frac{c\tau+d}{2iy}\right)^{k_0-k_1}
\overline{\theta_{L_1+\pi(\gamma-cz')}
\left(\tau, g_1;\begin{pmatrix} d\mu_{L_1}\\
-c\mu_{L_1}\end{pmatrix}; h_{\kappa^{(1)}}^{\bp-1}\right)}
y^\nu\frac{dxdy}{y^2}.
\end{align*}
Then the term of $c=d=0$ vanishes unless $k_0=k_1$, and hence equals
\begin{align*}
&\frac{\delta_{kk_1}|a|}{\sqrt{2}\Vert v_0^+(z)\Vert}
\int_{\SL_2(\Z)\backslash \mathbb{H}_1}
\sum_{\lambda\in D_{L_1}}
\Bigl(\sum_{\substack{\delta\in M'/L\\ \pi(\delta)=\lambda}}
f_{L+\delta}(\tau)\Bigr)\overline{\theta_{L_1+\lambda}
\bigl(\tau, g_1; h_{(k, k_2,\ldots,\pm k_{\bp-1})}^{\bp-1}\bigr)}
y^\nu\frac{dxdy}{y^2}\\
&\qquad =\frac{\delta_{kk_1}|a|}{\sqrt{2}\Vert v_0^+(z)\Vert}
\bigl(f_{L_1}(\tau),\Theta_{L_1}(\tau , g_1; h_{(k,k_2,\ldots, 
\pm k_{\bp-1})}^{\bp-1})\bigr)_\tau,
\end{align*}
when $\kappa$ is of form $\kappa=(k,k,\ldots)$ (in this case, $h^\bp_\kappa$ 
depends only on $\bp-1$ variables from the start, and we can identify 
$h^\bp_\kappa$ with $h^{\bp-1}_{\kappa^{(1)}}$ as polynomials).

To compute the remaining terms, we change $\gamma$ to $\gamma+cz'$ for each
pair $(c,d)$.
Then the integral of the sum over $(c,d)\neq (0,0)$ becomes
\begin{align*}
&\frac{a^{k_0-k_1}|a|}{\sqrt{2}(2i)^{k_0-k_1}\Vert v_0^+(z)\Vert^{k_0-k_1+1}}
\int_{\SL_2(\Z)\backslash \mathbb{H}_1}\sum_{(c,d)\neq (0,0)}
\be\left(\frac{-a^2|c\tau+d|^2}{4iy\Vert v_0^+(z)\Vert^2}\right)
(c\tau+d)^{k_0-k_1}\\
&\quad \times\sum_{\gamma\in D_{L_1}} 
f_{L_1+\gamma}\left(\tau; \begin{pmatrix} -d \\ c	 
			\end{pmatrix}\right)
\overline{\theta_{L_1+\gamma}\left(\tau,g_1; \begin{pmatrix} d\mu_{L_1} \\
                           -c\mu_{L_1} \end{pmatrix};
 h_{\kappa^{(1)}}^{\bp-1}\right)}y^{\nu-(k_0-k_1)}y^{-2}dxdy
\end{align*}
By replacing the sum over $(c,d)\neq (0,0)$ by a sum over $(nc,nd)$ with 
$c,d$ coprime and $n>0$, we get
\begin{align*}
&\frac{a^{k_0-k_1}|a|}{\sqrt{2}(2i)^{k_0-k_1}\Vert v_0^+(z)\Vert^{k_0-k_1+1}}
\int_{\SL_2(\Z)\backslash \mathbb{H}_1}\sum_{n=1}^\infty n^{k_0-k_1}
\sum_{(c,d)=1}
\be\left(\frac{-a^2n^2|c\tau+d|^2}{4iy\Vert v_0^+(z)\Vert^2}\right)
(c\tau+d)^{k_0-k_1}\\
&\qquad \times 
\left\langle f_{L_1}\left(\tau; \begin{pmatrix} -nd\\ nc \end{pmatrix}\right),
\Theta_{L_1}\left(\tau, g_1; 
\begin{pmatrix} nd\mu_{L_1} \\ -nc\mu_{L_1}\end{pmatrix}; 
h_{\kappa^{(1)}}^{\bp-1}\right)\right\rangle y^{\nu-(k_0-k_1)}
\frac{dxdy}{y^2}.
\end{align*} 
By Proposition \ref{thetatransformation} and (\ref{Bo5.3}), 
this becomes
\begin{align*}
&\frac{a^{k_0-k_1}|a|}{\sqrt{2}(2i)^{k_0-k_1}\Vert v_0^+(z)\Vert^{k_0-k_1+1}}
\int_{\SL_2(\Z)\backslash \mathbb{H}_1}
\sum_{n>0} n^{k_0-k_1}
\sum_{M\in \varGamma_\infty\backslash\SL_2(\Z)}
\be\left(\frac{-a^2n^2}{4i\mathrm{Im}(M\tau)
\Vert v_0^+(z)\Vert^2}\right)\\
&\times\left\langle f_{L_1}\left(M\tau; \begin{pmatrix} -n \\ 0
				    \end{pmatrix}\right),
\Theta_{L_1}\left(M\tau, g_1; 
\begin{pmatrix} n\mu_{L_1} \\ 0\end{pmatrix};
h_{\kappa^{(1)}}^{\bp-1}\right)\right\rangle 
\mathrm{Im}(M\tau)^{\nu-(k_0-k_1)}\frac{dxdy}{y^2}\\
=&\frac{2a^{k_0-k_1}|a|}{\sqrt{2}(2i)^{k_0-k_1}\Vert v_0^+(z)\Vert^{k_0-k_1+1}}
\int_0^\infty\sum_{n>0} n^{k_0-k_1}
\be\left(\frac{-a^2n^2}{4iy\Vert v_0^+(z)\Vert^2}\right)\\
&\qquad \times
\int_0^1\left\langle f_{L_1}\left(\tau; \begin{pmatrix} -n \\ 0
					  \end{pmatrix}\right),
\Theta_{L_1}\left(\tau, g_1; \begin{pmatrix} n\mu_{L_1} \\
                            0 \end{pmatrix};
h_{\kappa^{(1)}}^{\bp-1}\right)\right\rangle dxy^{\nu-(k_0-k_1)-2}dy.
\end{align*}
We insert the Fourier expansion of $f_{L_1}$ and  the definition
(\ref{generaltheta}) of theta series to compute the integral in $x$.
After that we use the formula
$$
\int_0^\infty e^{-Ay-B/y}y^{\nu-1}dy=2(B/A)^{\nu/2}K_{\nu}(2\sqrt{AB})
$$ 
to compute the integral in $y$.
After these steps the above expression becomes
\begin{align*}
&\frac{(-ia)^{k-k_1}|a|^{k_1+\frac{\bp-1}{2}}}
{2^{k+\frac{\bp}{2}-3}
\Vert v_0^+(z)\Vert^{k+\frac{\bp-1}{2}}}
\sum_{\substack{\lambda\in L_1'\\ q(\lambda)>0}}
\sum_{n=1}^{\infty}
n^{k+\frac{\bp-3}{2}}
\Bigl(\sum_{\substack{\delta\in M'/L\\ \pi(\delta)= \lambda}}
\be\bigl((n\delta, z')\bigr)
c(q(\lambda),\delta)\Bigr)\\
&\quad\quad\times \be\bigl((n\lambda, u+ag_1 (z^\ast)_{L_1})\bigr)
\frac{h_{\kappa^{(1)}}^{\bp-1}(v_1^+(g_1^{-1}\lambda))}
{\Vert v_1^+(g_1^{-1}\lambda)\Vert^{k_1+\frac{\bp-3}{2}}}
K_{k_1+\frac{\bp-3}{2}}
\left( \frac{2\pi n|a|\Vert v_1^+(g_1^{-1}\lambda)\Vert}
{\Vert v_0^+(z)\Vert}\right),
\end{align*}
which concludes the claimed identity.
\end{proof}
We focus on the factor 
\begin{equation}\label{l1lift}
\vartheta_{L_1}(f_{L_1}; h_{\kappa^{(1)}}^{\bp-1})(g_1)
=(f_{L_1}(\tau), \Theta_{L_1}(\tau,g_1; h_{\kappa^{(1)}}^{\bp-1}))_\tau
\end{equation}
occurring in the constant term of $\vartheta_L(f; h_\kappa^\bp)(g_z)$
in Theorem \ref{fourierexpansion>2}.
Let $\varGamma_{L_1}$ be the discriminant kernel in $\mathrm{SO}^+(L_1)$.
By definition (\ref{l1lift}) satisfies
\begin{lemma}\label{modularityl1lift} 
Suppose $\bp>2$ and $k=k_1$ for $\kappa$.
Then for every $\gamma\in \varGamma_{L_1}$, 
$w\in M_z\cap K_{v_0}$,
$$
\vartheta_{L_1}(f_{L_1}; h_{\kappa^{(1)}}^{\bp-1})(\gamma g_1 w)
=\vartheta_{L_1}(f_{L_1}; \tau_k(w)h_{\kappa^{(1)}}^{\bp-1})(g_1),
\quad g_1\in \OO((L_1)_\R).
$$
\end{lemma} 
On the other hand, our choice of the Levi subgroup $M_z\subset P_z$ 
guarantees 
\begin{lemma}\label{gammafunctorial}\cite[Proposition 3.2]{Zea}
One has
$$
\varGamma_L\cap M_z=\{m_z(1,\gamma)\mid \gamma\in \varGamma_{L_1}\}\simeq
\varGamma_{L_1}.
$$
\end{lemma}
These two lemmas explain a complete recursive aspect of 
the expansion formula in Theorem \ref{fourierexpansion>2}.

When $\bp=2$, the Fourier expansion of $\vartheta_L(f; h^2_{\pm k})(g_z)$ 
was computed in \cite{Bo} by performing the binomial expansion
of $h^2_{\pm k}(x)$ inside his formula in \cite[Theorem 7.1]{Bo}.
\begin{theorem}\cite[Theorem 4]{Od}, \cite[Theorem 14.3]{Bo}
\label{fourierexpansion=2}
Suppose $L$ has signature $(2,\bm)$ and $k\ge 2$.
Then $\vartheta_L(f; h^2_{\pm k})(g_z)$ at
$g_z=n_z(u)m_z(a,g_1)\in P_z$ is equal to
\begin{align*}
&\frac{2a^k}{\Vert v_0^+(z)\Vert^k}
\sum_{\substack{\lambda\in L_1',~ q(\lambda)>0\\ (v_1(\lambda),e_2)>0}}
\biggl(\sum_{n\vert\lambda} n^{k-1}
\sum_{\substack{\delta\in M'/L\\ \pi(\delta)=\mp\lambda/n}}
\be((n\delta, z'))c\left(\frac{q(\lambda)}{n^2}, \delta\right)
\biggr)\\
&\qquad\qquad\qquad\times\be((\mp\lambda, u+ag_1(z^\ast)_{L_1}))
\exp\left(-\frac{2\pi|a|
\Vert v_1^+(g_1^{-1}\lambda)\Vert}{\Vert v_0^+(z)\Vert}\right).
\end{align*}
\end{theorem}

\section{The square integrability} 
In this section we will establish 
the square integrability of $\vartheta_L(f)(g)$
on $\varGamma_L\backslash \OO(L_\R)$.
\subsection{Rational parabolic subgroups}
Let $P$ be any $\Q$-parabolic subgroup of $\OO(L_\R)$ with Levi decomposition
$P=MN$.
Then we consider the constant term
\begin{align}
\vartheta_L(f)_P(g)=\int_{\Gamma_L\cap N\backslash N}\vartheta_L(f)(ng)dn
\end{align}
of the Fourier expansion of $\vartheta_L(f)(g)$ along $P$.
It is again an $H_k^\vee$-valued function on $\OO(L_\R)$.
In particular if $P=P_z$, Section \ref{smallerlattice}, and $\bp>2$, then
we have seen that 
\begin{equation}\label{firststep}
\vartheta_L(f)_{P_z}(g_z)(h_\kappa^{\bp})=\frac{\delta_{kk_1}
|a|}{\sqrt{2}\Vert v_0^+(z)\Vert}
\vartheta_{L_1}(f_{L_1}; h_{\kappa^{(1)}}^{\bp-1})(g_1)
\end{equation}
for $g_z=n_z(u)m_z(a,g_1)\in P_z$ in Theorem \ref{fourierexpansion>2}.
On the other hand if $\bp=2$ and $k\ge 2$, then 
Theorem \ref{fourierexpansion=2} shows that
$\vartheta_L(f)_{P_z}=0$.

As Lemmas \ref{modularityl1lift} and \ref{gammafunctorial}
have suggested, we can repeat the argument in Section 5 
to compute the Fourier expansion of 
$\vartheta_{L_1}(f_{L_1}; h_{\kappa^{(1)}}^{\bp-1})(g_1)$, (\ref{firststep}),
as long as we find a primitive isotropic vector in $L_1$. 
This observation  motivates us to introduce
a set $\mathbf{L}=\{L_j\}_{j=0}^s$ of smaller even lattices in $L_\R$
satisfying the following properties:
\begin{itemize}
\item $L_0=L$;
\item  For $j\ge 0$, there exists a pair $(z_j, z_j')$ of 
a primitive isotropic vector $z_j\in L_j$ and 
$z_j'\in L_j'$ satisfying $(z_j,z_j')=1$ so that
$L_{j+1}=L_j\cap \langle z_j\rangle^\perp
\cap \langle z_j'\rangle^\perp$, 
where $L_j'=\{\lambda\in (L_j)_\Q\mid
(\lambda,\mu)\in \Z\hbox{ for all }\mu\in L_j\}$;
\item Consecutive inclusions $L_{j+1}\subset L_j$ by definition,
and $L_s$ has no isotropic vector.
\end{itemize}
Here are a few remarks.
Obviously $s\le \min\{\bp,\bm\}$.
By restricting $(~,~)$ to $L_j$,
we have identified a primitive isotropic vector of $L$
belonging to $L_j$ with a primitive isotropic vector of $L_j$.
Also we regard $z_j'$ as an element of $L_\Q$ 
by the inclusions $L_j'\subset (L_j)_\Q\subset L_\Q$.

For any isotropic subspace $U\subset L_\Q$ 
we denote by $P_U$ the maximal $\Q$-parabolic subgroup of 
$\OO(L_\R)$ that stabilizes $U$.
Now suppose $\mathbf{L}=\{L_j\}_{j=0}^s$,
and $z_j\in L_j$ and $z_j'\in L_j'$ for $0\le j\le s-1$ are given as above.
Then for each $0\le r\le s-1$ we take the maximal
$\Q$-parabolic subgroup $P_{U_r}$ of $\OO(L_\R)$
stabilizing $U_r=\langle z_0,\ldots,z_r\rangle_\Q$.
Besides it, we define a linear map
$\alpha$ from $U_r'=\langle z_0',\ldots, z_r'\rangle_\Q$ to $U_r$ by setting
$\alpha(z_j')=q(z_j')z_j$.
Note that $U_r'$ is a complement of $U_r^\perp$ in $L_\Q$,
then $U_r''=\{u'-\alpha(u')\mid u'\in U_r'\}$ gives another isotropic
subspace of $L_\Q$ (see \cite[Lemma 1.5]{Zea}).
Consequently $P_{U_r}$ admits the Levi decomposition
$P_{U_r}=M_{U_r}N_{U_r}$ with the Levi subgroup 
$M_{U_r}=P_{U_r}\cap P_{U_r''}$ and the unipotent radical $N_{U_r}$.

On the other hand we define $\Q$-parabolic subgroup
\begin{equation}\label{rationalparabolic}
P_{\le r}=\bigcap_{0\leq j\leq r} P_{U_j}
\end{equation}
of $\OO(L_\R)$ for every $0\le r\le s-1$.
In particular $P_{\le s-1}$ is a minimal $\Q$-parabolic subgroup of $\OO(L_\R)$,
and each $P_{\le r}$ is understood as the stabilizer of the sequence
$U_0\subset U_1\subset \cdots\subset U_r$.
It admits the Levi decomposition 
$P_{\le r}=M_{\le r}N_{\le r}$  with the Levi subgroup
\begin{equation}\label{leviofrationalparabolic}
M_{\le r}=\bigcap_{0\le j\le r}M_{U_j}.
\end{equation}

\subsection{A suitable choice of an isometry}
To discuss the square integrability of $\vartheta_L(f)(g)$
there is no matter to replace an isometry $v_0: (L_\R,q)\simeq (V,Q)$ 
with another one that is convenient for a given $\mathbf{L}=\{L_j\}_{j=0}^s$.
Actually for a given $\widetilde{v_0}: L_\R\simeq V$ 
one can find $h\in \OO(L_\R)$ so that $\widetilde{v_0}\circ h^{-1}=v_0$,
which naturally sends
$\vartheta_L(f)(g)$ to $\vartheta_L(f)(gh)$. 
Clearly the square integrabilities of these functions are equivalent.

Now suppose $\mathbf{L}=\{L_j\}_{j=0}^s$, and 
$z_j\in L_j$ and $z_j'\in L_j'$ for $0\le j\le s-1$ are given as before. 
Correspondingly we choose an isometry 
$v_\mathbf{L}: (L_\R,q)\simeq (V,Q)$ given as follows.
First we take isotropic vectors 
$z_j^\ast=z_j'-\alpha(z_j')$ for $0\le j\le s-1$, 
which satisfy $\bigl(z_j, z_j^\ast)=1$
and $\langle z_j,z_j'\rangle_\Q=\langle z_j, z_j^\ast\rangle_\Q$.
Then we take $2s$ elements
$$
z_j^\pm= \frac12(z_j\pm z_j^\ast), \quad 0\le j \le s-1,
$$ 
which are orthogonal to each other and $\Vert z_j^\pm\Vert=1$. 
We use them to give an ordered orthogonal basis of $L_\R$ 
by adding more elements of $(L_s)_\R$ as appropriate.
We define the isometry $v_\mathbf{L}$ by sending 
this basis to the standard ordered orthogonal basis $\{e_i\}$
of $V$.
More precisely, $v_\mathbf{L}$ is defined as
\begin{align}
&v_\mathbf{L}(z_j^+)=e_{j+1},\quad
v_\mathbf{L}(z_j^-)=e_{b-j},\qquad 0\le j\le s-1,  \label{choiceofvL}\\
&\quad\hbox{ and }\
v_\mathbf{L}((L_s)_\R)=\langle e_{s+1},\ldots, e_{b-s}\rangle\quad
(b=\bp+\bm). \notag
\end{align}

The hyperbolic planes $\langle z_j, z_j'\rangle=\langle z_j,z_j^\ast\rangle
=\langle z_j^+,z_j^-\rangle$, $0\le j\le r$, contribute to 
the orthogonal decomposition
\begin{align*}
L_\R=\Bigl(\bigoplus_{j=0}^r\langle z_j,z_j'\rangle\Bigr)\oplus 
(L_{r+1})_\R \label{U_jdecomposition}
\end{align*}
for every $0\le r\le s-1$.
Here note that they are all defined over $\Q$.
The restriction of $v_\mathbf{L}$ to $(L_j)_\R$ induces 
an isometry $v_j$ of $(L_j)_\R$ into $V$, $0\le j\le s$.
Again we write $v_j(x)=v_j^+(x)+v_j^-(x)$ for every $x\in (L_j)_\R$
according to the decomposition
$v_j((L_j)_\R)=(v_j((L_j)_\R)\cap V^+)\oplus 
(v_j((L_j)_\R)\cap V^-)$.

\subsection{The constant term along $P_{\le r}$}
Recall the parabolic subgroups $P_{\le r}$, (\ref{rationalparabolic}),
for $0\le r\le s-1$.
Each has a bijection
\begin{equation}\label{coordinaterationalparabolic}
P_{\le r}\simeq \Bigl(\prod_{j=0}^r(L_{j+1})_\R\Bigr)\times(\R^\times)^{r+1}
\times \OO((L_{r+1})_\R),\quad 0\le r\le s-1,
\end{equation}
given as follows.
Firstly, for each $0\le j\le r$,
we take a maximal $\Q$-parabolic subgroup 
$Q_{\langle z_j\rangle}$ of 
$\OO((L_j)_\R)$ stabilizing the line $\langle z_j\rangle$
in $(L_j)_\R$ (thus in particular 
$Q_{\langle z_0\rangle}=P_{\langle z_0\rangle}$).
It admits a Levi decomposition $Q_{\langle z_j\rangle}=M_{z_j}N_{z_j}$
with the Levi subgroup $M_{z_j}=Q_{\langle z_j\rangle}\cap Q_{\langle z_j^\ast\rangle}$ and
the abelian unipotent radical $N_{z_j}$.
We then take group isomorphisms
$m_{z_j}: \GL_1(\R)\times \OO((L_{j+1})_\R)\simeq M_{z_j}$ and 
$n_{z_j}: (L_{j+1})_\R\simeq N_{z_j}$ given 
in the same way as (\ref{bijectionlevi1}) and (\ref{bijectionlevi2}). 
Now to each element
$(u_1,\ldots,u_{r+1},a_0,\ldots,a_r,g_{r+1})$ of the 
product set in the right-hand side of (\ref{coordinaterationalparabolic}),
let us attach $g_j=n_{z_j}(u_{j+1})m_{z_j}(a_j,g_{j+1})\in 
Q_{\langle z_j\rangle}$ successively for $j$
in descending order from $r$ to $0$, which then terminate at
$g_0\in P_{\langle z_0\rangle}\cap P_{\le r}=P_{\le r}$.
This assignment gives the bijection  (\ref{coordinaterationalparabolic})
for each $r$.
We write those $g_j\in Q_{\langle z_j\rangle}$, $0\le j\le r$, attached to
$(u_1,\ldots, u_{r+1}, a_0,\ldots, a_r, g_{r+1})$ simply as 
$g_j=[u_{j+1},\ldots u_{r+1}, a_j,\ldots, a_r, g_{r+1}]$.
Thus in particular
\begin{equation}\label{coordinate}
g_0=[u_1,\ldots,u_{r+1}, a_0,\ldots,a_r,g_{r+1}]\in P_{\le r}.
\end{equation}
Conversely for $g_0\in P_{\le r}$ given in this form,
we set
\begin{equation}
(g_0)_j=[u_{j+1},\ldots u_{r+1}, a_j,\ldots, a_r, g_{r+1}]
\in Q_{\langle z_j\rangle}
\end{equation}
for every $0\le j\le r$.

Recall the basis $\{h_\kappa^\bp(x)\}_\kappa$ of $H_k=\H^{k,0}(V)$
in Proposition \ref{harmonicbasis}.
For each multi-index
$\kappa=(k, k_1, \ldots, k_j, \ldots, \pm k_{\bp-2})$ we set 
$\kappa^{(j)}=(k_j, k_{j+1},\ldots, \pm k_{\bp-2})$, which has length $\bp-j-1$.
On the other hand we set 
\begin{equation}\label{positivecone}
\Lambda_{j+1}^+=\{\lambda\in L_{j+1}'\mid q(\lambda)>0 \}
\end{equation}
for every $0\le j\le s-1$, whose elements correspond 
to characters on the abelian unipotent radical $N_{z_j}$.
Now let us fix $r$, $0\le r\le s-1$, and
consider the constant term $\vartheta_L(f)_{P_{\le r}}(g)$ along $P_{\le r}$.
\begin{proposition}\label{constanttermrationalparabolic}
Let $g_0=[u_1,\ldots,u_{r+1}, a_0,\ldots,a_r ,g_{r+1}]\in P_{\le r}$.
Then the $H_k^\vee$-valued function $\vartheta_L(f)_{P_{\le r}}(g_0)$,
$f\in S_\nu(D_L)$,
satisfies the following properties.

\medskip
\noindent
{\rm (i)} If $k>k_{r+1}$ for $\kappa$, then 
$\vartheta_L(f)_{P_{\le r}}(g_0)(h_\kappa^\bp)=0$.

\medskip
\noindent
{\rm (ii)} On the other hand if $k=k_{r+1}$ for $\kappa$, then 
\begin{equation}\label{formularational}
\vartheta_L(f)_{P_{\le r}}(g_0)(h_\kappa^\bp)=
\left(\prod_{j=0}^r\frac{|a_j|}{\sqrt{2}\Vert v_j^+(z_j)\Vert}\right)\cdot
\vartheta_{L_{r+1}}\bigl(f_{L_{r+1}}; h_{\kappa^{(r+1)}}^{\bp-r-1}\bigr)
(g_{r+1}),
\end{equation}
where $f_{L_{r+1}}(\tau)$ denotes the form 
$(\cdots(f_{L_1})_{L_2}\cdots)_{L_{r+1}}
(\tau)\in S_\nu(D_{L_{r+1}})$.

\medskip
\noindent
{\rm (iii)} If $k=k_{r+1}$ for $\kappa$, then the Fourier expansion of 
$\vartheta_{L_j}(f_{L_j}; h_{\kappa^{(j)}}^{\bp-j})((g_0)_j)$ 
along $Q_{\langle z_j\rangle}$, $0\le j\le r$, is supported only on 
the characters of $N_{z_j}$ 
corresponding to $\lambda\in \Lambda_{j+1}^+\cup \{0\}$.
\end{proposition}
\begin{proof}
The proof is based on the identity (\ref{firststep}).
Then the first two assertions are obtained by repeatedly applying
Theorem \ref{fourierexpansion>2}.
The third assertion is directly deduced from 
the Fourier expansion formula in the same theorem.
\end{proof}

\subsection{The constant term along $P_{U_r}$}
Recall the maximal $\Q$-parabolic subgroup $P_{U_r}=M_{U_r}N_{U_r}$, 
$0\le r\le s-1$.
Let $\hbox{Hom}^{as}((U_r')_\R, (U_r)_\R)$ be the space of
$\R$-linear maps $\psi$ from $(U_r')_\R$ to $(U_r)_\R$ 
of the form $\psi(u')=\sum_{i,j}b_{ij}(z_j, u')z_i$
with $b_{ij}=-b_{ji}\in\R$.
Then there is a bijection 
\begin{align}\label{coordinatemaximalparabolic}
P_{U_r}&\simeq \GL((U_r)_\R)\times
\OO((L_{r+1})_\R)\times (L_{r+1})^{r+1}\times
\hbox{Hom}^{as}((U_r')_\R, (U_r)_\R);\\
g_{U_r}&\mapsto\ (A, g_{r+1}, \mathbf{x}, \psi);\quad 
\mathbf{x}=(x_0,x_1,\ldots, x_r)\notag
\end{align}
(see \cite[p.106, (7)]{Zea}), which is characterized by
\begin{align*}
&g_{U_r} u=A u\in (U_r)_\R,\quad 
g_{U_r}y=g_{r+1}y-\Bigl(\sum_i(x_i,g_{r+1}y)z_i\Bigr)
\in (L_{r+1})_\R\oplus (U_r)_\R\\
&g_{U_r} u'=\Bigl(A\alpha(u')-\alpha({}^tA^{-1}u')-\sum_{0\le i,j\le r}
\tfrac12(x_j,x_i)(z_i,{}^tA^{-1}u')z_j
+\psi({}^tA^{-1}u')\Bigr)\\
&\qquad\qquad\qquad+\Bigl(\sum_{0\le i\le r}(z_i,{}^tA^{-1}u')x_i\Bigr)
+{}^tA^{-1}u'
\in (U_r)_\R\oplus (L_{r+1})_\R\oplus(U_r')_\R
\end{align*}
for $u\in (U_r)_\R$, $y\in (L_{r+1})_\R$, and $u'\in (U_r')_\R$.
Here ${}^tA$ denotes the unique element in 
$\GL((U_r')_\R)$ satisfying $(Au,u')=(u,{}^tAu')$ 
for every $u\in U_r$ and $u'\in U_r'$.
Using this bijection we write then 
$$
g_{U_r}=[A,g_{r+1}, \mathbf{x},\psi]\in P_{U_r}.
$$
We remark that $M_{U_r}=\{[A,g_{r+1},0,0]\}$ and that
$N_{U_r}=\{[\mathrm{id}_{U_r},\mathrm{id}_{L_{r+1}},\mathbf{x},\psi]\}$,
a generalized Heisenberg group with the center 
\begin{equation}\label{centerheisenberg}
Z(N_{U_r})=\{[\mathrm{id}_{U_r},\mathrm{id}_{L_{r+1}},0,\psi]\}
\end{equation}
and $N_{U_r}/Z(N_{U_r})\simeq (L_{r+1})_\R^{r+1}$.
According to \cite[Corollary 3.3]{Zea} we see that
\begin{align*}
&\varGamma_L\cap N_{U_r}\\
&\quad =\left\{[\mathrm{id}_{U_r},\mathrm{id}_{(L_{r+1})_\R},\mathbf{x},\psi]
\mid \mathbf{x}\in (L_{r+1})^{r+1} \hbox{ and }b_{ij}+\frac{(x_i,x_j)}{2}\in \Z
\hbox{ for }0\le i,j\le r\right\}.
\end{align*}

We can observe the following relations between the expressions 
(\ref{coordinaterationalparabolic}) and
(\ref{coordinatemaximalparabolic}):
\begin{itemize}
 \item  Any element $g_0\in P_{\le r}$ corresponds to 
$[A, g_{j+1}, \mathbf{x},\psi]\in P_{U_r}$ 
with $A$ of ``upper triangular type'' in the certain sense of the word 
through the inclusion $P_{\le r}\subset P_{U_r}$.

\item
$[\mathrm{id}_{U_r},\mathrm{id}_{(L_{r+1})_\R},0,\psi]=
[\psi(z_0'), \ldots, \psi(z_r'), 1,\dots,1, 
\mathrm{id}_{(L_{r+1})_\R}]\in Z(N_{U_r})$
through the inclusions $Z(N_{U_r})\subset N_{U_r}\subset N_{\le r}$.

\item
$[\mathrm{id}_{U_r},\mathrm{id}_{(L_{r+1})_\R},\mathbf{x},0]
=\bigl[\mathbf{x}-\tfrac12\varphi(\mathbf{x}),1,\dots,1,
\mathrm{id}_{(L_{r+1})_\R}\bigr]
\in N_{U_r}$, where  we set 
$$
\varphi(\mathbf{x})=\Bigl(\sum_{1\le i\le r}(x_i,x_0)z_i,\ldots, 
\sum_{j+1\le i\le r}(x_i,x_j)z_i,\ldots, (x_r,x_{r-1})z_r, 0\Bigr)\in 
\prod_{j=0}^r(L_{j+1})_\R.
$$
for $\mathbf{x}=(x_i)_{i=0}^r\in 
(L_{r+1})_\R^{r+1}\subset\prod_{j=0}^r(L_{j+1})_\R$.
\end{itemize}
\begin{proposition}\label{constanttermmaximalparabolic}
Let $g_0=[a\cdot A_0, g_{r+1},\mathbf{x},\psi]
\in P_{\le r}\subset P_{U_r}$ with $a\in\R^\times$ and $A_0\in \SL((U_r)_\R)$
of upper triangular type.
Then the $H_k^\vee$-valued function $\vartheta_L(f)_{P_{U_r}}(g_0)$,
$f\in S_\nu(D_L)$,
satisfies the following properties.

\medskip
\noindent
{\rm (i)} If $k>k_{r+1}$ for $\kappa$, 
then $\vartheta_L(f)_{P_{U_r}}(g_0)(h_\kappa^{\bp})=0$.

\medskip
\noindent
{\rm (ii)} On the other hand if $k=k_{r+1}$ for $\kappa$, then
\begin{equation}\label{formulamaximal}
\vartheta_L(f)_{P_{U_r}}(g_0)(h_\kappa^{\bp})= 
\left(\prod_{j=0}^r\frac{1}{\sqrt{2}\Vert v_j^+(z_j)\Vert}\right)
|a|^{r+1} \vartheta_{L_{r+1}}(f_{L_{r+1}};
h_{\kappa^{(r+1)}}^{\bp-r-1})(g_{r+1}).
\end{equation}
\end{proposition}
\begin{proof}
Proposition \ref{constanttermrationalparabolic} (iii) implies 
that we can write $\vartheta_L(f)(g_0)$ in the form
\begin{equation}\label{roughexpansion}
\vartheta_L(f)(g_0)=|a|^{r+1}\Phi^{L_r}_{0_{L_{r+1}}}(g_{r+1})+
\sum_{j=0}^r \sum_{\lambda\in \Lambda_{j+1}^+}|a|^j
\Phi^{L_j}_\lambda ((g_0)_j),
\end{equation}
where $\Phi^{L_j}_\lambda((g_0)_j)$, $(g_0)_j\in Q_{\langle z_j\rangle}$,
denotes the gathering of
the $\lambda$-th Fourier coefficient of 
$\vartheta_{L_j}(f_{L_j}; h_{\kappa^{(j)}}^{\bp-j})((g_0)_j)$
along $Q_{\langle z_j\rangle}$
over the various $\kappa^{(j)}$'s.
In particular we observe 
$$
|a|^{r+1}\Phi^{L_r}_{0_{L_{r+1}}}(g_{r+1})
=\vartheta_L(f)_{P_{\le r}}(g_0),
$$
which has been described in Proposition
\ref{constanttermrationalparabolic} (ii).

Now we put the expression (\ref{roughexpansion}) into the right-hand side
of 
$$
\vartheta_L(f)_{P_{U_r}}(g_0)=\int_{Z(N_{U_r})\varGamma_L\cap N_{U_r}
\backslash N_{U_r}}\int_{\varGamma_L\cap Z(N_{U_r})\backslash Z(N_{U_r})}
\vartheta_L(f)(n_1n_2g_0)dn_2dn_1,
$$
and look at every integral
\begin{equation}\label{lambdapart}
\int_{Z(N_{U_r})\varGamma_L\cap N_{U_r}
\backslash N_{U_r}}\int_{\varGamma_L\cap Z(N_{U_r})\backslash Z(N_{U_r})}
\Phi^{L_j}_\lambda((n_1n_2g_0)_j)dn_2dn_1
\end{equation}
for $0\le j\le r$ and $\lambda\in \Lambda_{j+1}^+$.
Firstly we focus on the condition for the inner integral over 
$\varGamma_L\cap Z(N_{U_r})\backslash Z(N_{U_r})$ not to vanish.
Using the coordinate (\ref{centerheisenberg}) of $Z(U_r)$ 
we then find that it is necessary that $\psi(\lambda)=0$ for all 
$\psi\in\hbox{Hom}^{as}((U_r')_\R, (U_r)_\R)$, or equivalently 
it should be satisfied that 
\begin{equation}\label{centercondition}
\lambda\in (U_r)_\R^\perp\cap L_{j+1}'
=((U_r)_\R\oplus (L_{r+1})_\R)\cap L_{j+1}'. 
\end{equation}
When this holds, next we consider the condition for the integral of  
$\Phi^{L_j}_\lambda((n_1g_0)_j)$ 
over $Z(N_{U_r})\varGamma_L\cap N_{U_r}
\backslash N_{U_r}\simeq (L_{r+1}\backslash (L_{r+1})_\R)^{r+1}$ not to vanish.
Then we find that it should be satisfied that
\begin{equation}\label{outercondition}
\lambda\in (L_{r+1})_\R^\perp\cap L'_{j+1}. 
\end{equation}
To sum up, (\ref{centercondition}) as well as (\ref{outercondition})
are necessary for the nonvanishing of the integral (\ref{lambdapart}),
which is equivalent to that 
$\lambda\in (U_r)_\R\cap L_{j+1}'$ and 
$\lambda\in \Lambda_{j+1}^+$ at the same time.
However this is impossible, because $U_r$ is isotropic.
Therefore we conclude that
$$
\vartheta_L(f)_{P_{U_r}}(g_0)=\vartheta_L(f)_{P_{\le r}}(g_0),\quad 
g_0\in P_{\le r},
$$
which gives the claimed assertion.
\end{proof}
\begin{corollary}\label{corollaryofparabolic}
For any $\Q$-parabolic subgroup $P$ such that
$P_{\le r}\subset P\subset P_{U_r}$ one has
\begin{align}
\vartheta_L(f)_P(g)=\vartheta_L(f)_{P_{\le r}}(g).
\end{align}
\end{corollary}
\begin{proof}
We observe that
$\vartheta_L(f)_P(g)=(\vartheta_L(f)_{P_{U_r}})_P(g)$.
Since the unipotent radical of $P$ is contained in $N_{\le r}$,
the assertion follows from the proof of the proposition.
\end{proof}

\subsection{The square integrability of the lifting}
In order to show the square integrability of $\vartheta_L(f)(g)$, 
we will apply the criterion in \cite[Lemma I.4.11]{MW} 
to the formulas of constant terms in our previous results, Propositions
\ref{constanttermrationalparabolic}, \ref{constanttermmaximalparabolic}, 
and Corollary \ref{corollaryofparabolic}.
\begin{theorem}\label{theoremonsquareintegrability}
Let $f(\tau)\in S_\nu(D_L)$.
Then $\vartheta_L(f)(g)$ is square integrable on $\varGamma_L\backslash
\OO(L_\R)$, except when $\OO(L_\R)\simeq \OO(3,1)$ 
and $\vartheta_{L_1}(f_{L_1}; h^2_{\pm k})(g_1)\neq 0$
for a positive definite lattice $L_1\subset L$ of rank $2$.
Here one notes that
$$
\theta_{L_1+\gamma}(\tau, g_1; h^2_{\pm k})=\sum_{\lambda\in L_1+\gamma}
(g_1^{-1}\lambda, e_2\pm ie_3)^k\be(q(\lambda)\tau)
$$
for $g_1\in \mathrm{O}((L_1)_\R)\simeq\mathrm{O}(2)$,
and $\gamma\in D_{L_1}$, and $\langle e_2,e_3\rangle=(L_1)_\R$.
\end{theorem}
\begin{proof}
Recall Proposition \ref{constanttermmaximalparabolic} and 
the formula (\ref{formulamaximal}).
In accordance with \cite[I.217, p.38, (v)]{MW}, 
we write  $|a|^{r+1}=\chi_{P_{U_r}}(a)|a|^{(b-r-2)(r+1)/2}$ by putting
$$
\chi_{P_{U_r}}(a)=|a|^{-(b-r-4)(r+1)/2}.
$$
Then this character $\chi_{P_{U_r}}$ exactly gives 
the cuspidal exponent of $\vartheta_L(f)(g)$ along $P_{U_r}$.
We see that if $b-r-4>0$, then
$\chi_{P_{U_r}}$ satisfies the condition in \cite[Lemma I.4.11]{MW}.

Next we recall Proposition \ref{constanttermrationalparabolic}
and (\ref{formularational}) for the description of 
$\vartheta_L(f)_{P_{\le r}}(g_0)(h_\kappa^{\bp})$
at $g_0=[u_1,\ldots,h_{r+1},a_0,\ldots,a_r,g_{r+1}] \in P_{\le r}$ for
every $0\le r\le s-1$.
Then we write $\prod_{j=0}^r|a_j|
=\chi_r(\mathbf{a})\prod_{j=0}^r|a_j|^{(b-2j-2)/2}$ by putting
$$
\chi_r(\mathbf{a})=\prod_{j=0}^r |a_j|^{-(b-2j-4)/2},\quad
\mathbf{a}=(a_0,\ldots,a_r)\in(\R^\times)^{r+1}.
$$
It can be also written as
$$
\chi_r(\mathbf{a})=\left(\prod_{j=0}^{s-3}
\left|\frac{a_j}{a_{j+1}}\right|^{-\frac{(j+1)(b-j-4)}{2}}\right)\cdot
\left|\frac{a_{s-2}}{a_{s-1}}\right|^{-\frac{s(b-s-3)}{4}+1}
\cdot |a_{s-2}a_{s-1}|^{-\frac{s(b-s-3)}{4}}
$$
if $\bp=\bm=s=r+1$, or 
$$
\chi_r(\mathbf{a})=\left(\prod_{j=0}^{r-1}
\left|\frac{a_j}{a_{j+1}}\right|^{-\frac{(j+1)(b-j-4)}{2}}\right)\cdot
|a_r|^{-\frac{(r+1)(b-r-4)}{2}},\hbox{ otherwise}.
$$
This character $\chi_r$ exactly gives 
the cuspidal exponent of $\vartheta_L(f)(g)$ along $P_{\le r}$.
We see that $\chi_r$ satisfies the condition in \cite[Lemma I.4.11]{MW}, if
$b-r-4>0$ again.
In particular the cuspidal exponent $\chi_{s-1}$ along
$P_{\le s-1}$ satisfies the condition, when $b-s-3>0$. 

Now suppose the condition that $b-s-3>0$ is satisfied.
Then Corollary \ref{corollaryofparabolic} ensures that
the cuspidal exponent along any $\Q$-parabolic subgroup $P$ 
satisfies the condition in \cite[Lemma I.4.11]{MW}. 
Thus we can conclude the square integrability of $\vartheta_L(f)(g)$ on
$\varGamma_L\backslash \OO(L_\R)$ in this case.

Suppose on the other hand that $b-s-3\le 0$.
Since $b\ge 2s$, we then get $\OO(L_\R)$ is
isomorphic to the split $\mathrm{O}(s,s)$ with $s=2, 3$, or
the split $\OO(s+1,s)$ with $s=1, 2$, or the quasi-split 
$\OO(3,1)$.
In the case that $\OO(L_\R)\simeq \OO(s+1,s)$, we see  
$\OO((L_{s-1})_\R)\simeq \OO(2,1)$.
Then Theorem \ref{fourierexpansion=2} implies that 
$\vartheta_L(f)_{P_{\le s-1}}(g)=0$,
and thus we conclude the square integrability of $\vartheta_L(f)(g)$ by
applying \cite[Lemma I.4.11]{MW} again.
In the case that $\OO(L_\R)\simeq \OO(s,s)$,
we have $\OO((L_{s-2})_\R)\simeq \OO(2,2)$.
Again we can apply Theorem \ref{fourierexpansion=2} to conclude the claimed
square integrability in this case.
Finally, when $\mathrm{O}(L_\R)\simeq\mathrm{O}(3,1)$, 
we get the statement in the theorem.
This finishes the proof of the theorem.
\end{proof}
\section{Archimedean Bessel integrals} 
Let $P_z$ be the maximal $\Q$-parabolic subgroup of $\OO(L_\R)$ 
corresponding to an isotropic vector $z\in L$, Section \ref{smallerlattice}.
For every $g_z\in n_z(u)m_z(a,g_1)\in P_z$ and $\lambda\in (L_1)_\R$
with $q(\lambda)>0$, we set
\begin{align}\label{archimedeancoefficient}
W_\lambda^\kappa(g_z)&=
 \mathbf{e}\bigl((\lambda, u)\bigr)
\frac{(-i)^{k-k_1}(a/|a|)^{k-k_1}|a|^{k+\frac{\bp-1}{2}}}
{2^{k_1+\frac{\bp}{2}-3}(k-k_1)!\Vert v_0^+(z)\Vert^{k+\frac{\bp-1}{2}}}\\
&\qquad \times \frac{h_{\kappa^{(1)}}^{\bp-1}(v_1^+(g_1^{-1}\lambda))}
{\Vert v_1^+(g_1^{-1}\lambda)\Vert^{k_1+\frac{\bp-3}{2}}}
K_{k_1+\frac{\bp-3}{2}}
\left( \frac{2\pi |a|\Vert v_1^+(g_1^{-1}\lambda)\Vert}
{\Vert v_0^+(z)\Vert}\right),\notag
\end{align}
if $\bp>2$, and
\begin{equation}
 W_\lambda^{\pm k}(g_z)=\be((\lambda, u))
\frac{2a^k}{\Vert v_0^+(z)\Vert^k}\exp\left(\frac{2\pi|a|
\Vert v_1^+(g_1^{-1}\lambda)\Vert}{\Vert v_0^+(z)\Vert}\right),
\ \hbox{ if } \bp=2.
\end{equation}

Recall the degenerate principal series representation 
$I_\varepsilon(k+\bp-2)$ of $\OO(V)$ 
and its irreducible submodule $\Pi^{\bp,\bm}_{k,0}$, 
Proposition \ref{cohomologicalrep}.
Depending on the decomposition $V=V_1\oplus\langle z_0,z_0^\ast\rangle$
with $z_0$ and $z_0^\ast$ in
(\ref{standardisotropic}), we define $w\in \OO(V)$ so that 
$wz_0=z_0^\ast$, $wz_0^\ast=z_0$, and $w|_{V_1}=\mathrm{id}_{V_1}$.
Now for every $f(g; h)$, (\ref{basisofminimalK}),
belonging to the minimal $K$-type of $\Pi^{\bp,\bm}_{k,0}$
we define the Bessel integral
\begin{equation}\label{besselintegral}
W_\lambda(g; h)=\int_{V_1} 
\be(-(\lambda,u))f(wn_0(u)g; h)du,\quad g\in\OO(V)\hbox{ and }\lambda\in V_1.
\end{equation}
In order to confirm that the lifting $\vartheta_L(f)(g)$
belongs to $\Pi^{\bp,\bm}_{k,0}$, 
we will compute the inverse Fourier transform of 
(\ref{archimedeancoefficient}) by following the arguments in 
\cite{KM} and \cite{Po}.
Then we will show that (\ref{archimedeancoefficient})
expresses the values of (\ref{besselintegral}).

\subsection{Non degenerate terms} 
When $\bp>2$, each $h_\kappa^\bp(x)\in \H^{k,0}(V)$
can be written as
$$
h_\kappa^\bp(x)=\Vert x\Vert^{k-k_1}
C_{k-k_1}^{k_1+\frac{\bp-2}{2}}\left(\frac{x_1}{\Vert x\Vert}\right)
h_{\kappa^{(1)}}^{\bp-1}(u^+)
$$
for $x=(x_1,u^+)\in V^+=\langle e_1\rangle\oplus V_1^+$ 
with $u^+\in V_1^+$.
In contrast to this we treat 
$$
h^2_{\pm k}(x)=(x, e_1\pm ie_2)^k
$$ 
as it is, when $\bp=2$.
Again we write $u=u^++u^-\in V_1^+\oplus V_1^-$ for each $u\in V_1$.
Define a positive valued function $\tau(u)\equiv\tau(u^+,u^-)
=\tau(u^-,u^+)$ on $V_1$ by
\begin{align*}
\tau(u)^2
=\tau(u^+,u^-)^2
=&~\left(1-\frac{\Vert u^+\Vert^2-\Vert u^-\Vert^2}{4}\right)^2
+\Vert u^+\Vert^2=\left(1-\frac{Q(u)}{2}\right)^2+\Vert u^+\Vert^2\\
=&~\left(1+\frac{\Vert u^+\Vert^2-\Vert u^-\Vert^2}{4}\right)^2
+\Vert u^-\Vert^2=\left(1+\frac{Q(u)}{2}\right)^2+\Vert u^-\Vert^2.
\end{align*}
(see \cite[(2.8.1)]{KO}).
In particular when $\bm=1$, hence $V_1=V_1^+$ and $V_1^-=\{0\}$, 
we recognize it as
$$
\tau(u)=1+\tfrac12 Q(u),\quad u=u^+\in V_1^+.
$$

Here let us explicitly take the isomorphisms $n_0$ and $m_0$ in
(\ref{leviisomorphismsV}) as 
\begin{align}\label{explicitleviisomorphismV}
&n_0(u)z_0=z_0,\quad n_0(u)x=-(u,x)z_0+x,\quad 
n_0(u)z_0^\ast=-Q(u)z_0+u+z_0^\ast, \\
&m_0(a,g_1)z_0=az_0,\quad m_0(a, g_1)x=g_1x,\quad 
m_0(a,g_1)z_0^\ast=a^{-1}z_0^\ast\label{explicitleviisomorphismVm}
\end{align}
for every $u,x\in V_1$, $a\in \R^\times$, and $g_1\in\OO(V_1)$.
\begin{lemma}\label{besselsection} Suppose $\bp>2$ and $\bm\ge 1$.
Then
$$
f(wn_0(u); h_\kappa^\bp)=
\frac{h_{\kappa^{(1)}}^{\bp-1}(u^+)}{\tau(u)^{k+k_1+\bp-2}}
C_{k-k_1}^{k_1+\frac{\bp-2}{2}}
\left(\frac{1-\frac12Q(u)}{\tau(u)}\right)
$$
for $u=u^++u^-\in V_1$ and $h_\kappa^\bp(x)\in\H^{k,0}(V)$.
On the other hand when $\bp=2$ and $\bm\ge 1$, one has 
$$
f(wn_0(u); h_{\pm k}^2)
=\left(1-\tfrac12 Q(u)\mp i(u,e_2)\right)^{-k}
=\left(e_2\mp i\frac{u}{2}, e_2\mp i\frac{u}{2}\right)^{-k}.
$$
\end{lemma}
\begin{proof}
Since $(wn_0(u))^{-1}z_0=n_0(-u)z_0^\ast=-Q(u)z_0-u+z_0^\ast$,
we see that 
$$
\Vert \bigl((w n_0(u))^{-1}z_0\bigr)^+\Vert=\tau(u)
\ \hbox{ and }\ \bigl((wn_0(u))^{-1}z_0, e_1\bigr)=1-\tfrac12Q(u).
$$
Besides these, we note that $V_1^+=\langle e_2\rangle\subset
V^+=\langle e_1,e_2\rangle$ when $\bp=2$, and 
$$
f(g; h_{\pm k}^2)=\frac{(g^{-1}z_0,e_1\pm ie_2)^k}
{\Vert (g^{-1}z_0)^+\Vert^{2k}}=(g^{-1}z_0, e_1\mp ie_2)^{-k}.
$$
The lemma follows from these identities.
\end{proof}
Note that $\Vert z_0^+\Vert=\Vert e_1/2\Vert=1/2$.
Assume that $\bp>2$ and $\bm\ge 1$.
Then for every $h_\kappa^\bp(x)\in \H^{k,0}(V)$ we define a function 
\begin{equation}
J^\kappa(\lambda)
=\mathrm{char}(Q(\lambda)>0)\Vert \lambda\Vert^{2A}
\frac{h_{\kappa^{(1)}}^{\bp-1}(\lambda^+)}
{\Vert \lambda^+\Vert^{k_1+\frac{\bp-3}{2}}}
K_{k_1+\frac{\bp-3}{2}}(4\pi\Vert \lambda^+\Vert),\quad
\lambda\in V_1,
\end{equation}
on $V_1$, where $\mathrm{char}(Q(\lambda)>0)$ means 
the characteristic function 
of the subset of all $\lambda\in V_1$ with $Q(\lambda)>0$ and 
$A=k+\frac{\bp-\bm}{2}-1$.
We consider its transform
\begin{equation}\label{inversefourier}
I^\kappa(u)=\int_{V_1}\be((\lambda,u))J^\kappa(\lambda)d\lambda,
\quad\ u\in V_1, 
\end{equation}
with the standard Lebesgue measure $d\lambda$ on $V_1\simeq \R^{b-2}$.

On the other hand when $\bp=2$ and $\bm\ge 1$, we define
\begin{align}
J^{\pm k}(\lambda)&=(-1)^k\cdot\mathrm{char}(Q(\lambda)>0)
\Vert \lambda\Vert^{2k-\bm}
\sum_{h=0}^k\begin{pmatrix} k\\ h \end{pmatrix} (-i)^h \\
&\times \sum_{j=0}^{[\frac{k-h}{2}]}
\frac{(k-h)!(\lambda,\pm ie_2)^{k-h-2j}}
{(8\pi)^jj!(k-h-2j)!\Vert\lambda^+\Vert^{k-h-j-\frac12}}
K_{k-h-j-\frac12}(4\pi \Vert \lambda^+\Vert) \notag
\end{align}
for every $\lambda\in V_1=\langle e_2\rangle$.
We have then 
\begin{equation}\label{J2}
J^{\pm k}(\lambda)=2^{k-\frac32}i^k\mathrm{char}\bigl(Q(\lambda)>0,
(\pm\lambda, e_2)<0\bigr)
\Vert \lambda\Vert^{2k-\bm}\mathrm{exp}(4\pi (\pm\lambda, e_2))
\end{equation}
(see the proof of \cite[Theorem 14.2]{Bo}).
We also consider the integral transforms
$I^{\pm k}(u)$ of $J^{\pm k}(\lambda)$
defined by replacing $J^\kappa(\lambda)$ with 
$J^{\pm k}(\lambda)$ in (\ref{inversefourier}).
\begin{lemma}
The integrals $\int_{V_1}|J^\kappa(\lambda)|d\lambda$ 
and $\int_{V_1}|J^{\pm k}(\lambda)|d\lambda$ are finite.
\end{lemma}
\begin{proof}
This can be proved in the same way as in the proof of \cite[Lemma 4.5.2]{Po}.
\end{proof}
According to this lemma, it makes sense to apply the Fourier inversion theorem
to the pair $J^\kappa(\lambda)$ and $I^\kappa(u)$, or $J^{\pm k}(\lambda)$
and $I^{\pm k}(u)$.
So let us compute the value of $I^\kappa(u)$ ($\bp>2$), or $I^{\pm k}(u)$
($\bp=2$).
We recall the hypergeometric series
$$
F_4(a,b;c,d;x,y)=\sum_{m,n=0}^\infty
\frac{(a)_{m+n}(b)_{m+n}}{(c)_{m}(d)_{n}m!n!}x^my^n.
$$
\begin{lemma}\label{inversefourier2}
{\rm (i)} Assume $\bp> 2$ and $\bm>1$. 
Then $I^\kappa(u)$,
$u=u^++u^-\in V_1$, equals
\begin{align*}
&\frac{i^{k_1}\Gamma(k+k_1+\bp-2)
\Gamma\bigl(k+\frac{\bp-\bm}{2}\bigr)}
{2^{k_1+\frac{\bp+\bm}{2}}(2\pi)^{2k+\bp-\frac{\bm+3}{2}}
\Gamma\bigl(k_1+\frac{\bp-1}{2}\bigr)}
     h_{\kappa^{(1)}}^{\bp-1}(u^+)\\
&\times  F_4\left(k+\frac{\bp-1}{2},k+k_1+\bp-2;~ k+\frac{\bp-1}{2},
k_1+\frac{\bp-1}{2};~-\frac{\Vert u^-\Vert^2}{4},
-\frac{\Vert u^+\Vert^2}{4}\right).
\end{align*}

\medskip
\noindent
{\rm (ii)} Assume $\bp>2$ and $\bm=1$.
Then $I^\kappa(u)$, $u=u^+\in V_1^+$, equals
\begin{align*}
&\frac{i^{k_1}\Gamma(k+k_1+\bp-2)
\Gamma\bigl(k+\frac{\bp-1}{2}\bigr)}
{2^{k_1+\frac{\bp+1}{2}}(2\pi)^{2k+\bp-2}\Gamma(k_1+\frac{\bp-1}{2})}
h_{\kappa^{(1)}}^{\bp-1}(u) \\
&\qquad\qquad\times 
{}_2F_1\left(k+k_1+\bp-2, k+\frac{\bp-1}{2};~ k_1+\frac{\bp-1}{2};~
-\frac{\Vert u\Vert^2}{4}\right).
\end{align*}
\end{lemma}
\begin{proof}
We follow the idea in the proof of \cite[Proposition 4.5.3]{Po}.
Let $S(V_1^\pm)$ denote the unit sphere in $V_1^\pm$.
Firstly we treat the case that $\bm>1$, thus $V_1^-\neq\{0\}$.
We write $\lambda=\tp\sigmap+\tm\sigmam$ with $t_\pm>0$ and $\sigma_\pm\in 
S(V_1^\pm)$.
Then $I^\kappa(u)$ is equal to
\begin{align*}
\int_{\tp>\tm>0} \int_\sigmap \int_\sigmam&
\be(\tp(\sigmap, u^+))\be(\tm(\sigmam, u^-))
(t_+^2-t_-^2)^A t_+^{-\frac{\bp-3}{2}}\\
&\times h_{\kappa^{(1)}}^{b^+-1}(\sigmap)
     K_{k_1+\frac{\bp-3}{2}}(4\pi \tp) t_+^{\bp-2} t_-^{\bm-2}
      d\tp d\tm d\sigmap d\sigmam.
\end{align*}
We recall the formula \cite[p.55 (3,3,4)]{KM}
\begin{align}\label{sphericalfourier}
\int_{S^{m-1}}e^{i t(u,\sigma)}\phi(\sigma)d\sigma
=(2\pi)^{\frac{m}{2}}i^{l}
\phi(u)t^{1-\frac{m}{2}}J_{\frac{m}{2}-1+l}(t),\quad u\in S^{m-1},
\end{align}
where $\phi$ is harmonic polynomial of degree $l$ and
$J_m$ denote the $J$-Bessel function.
Using this identity and putting 
$\tm=s\tp$ with $0<s<1$, we can express $I^\kappa(u)$ as
\begin{align*}
&\frac{(2\pi)^2 i^{k_1}h_{\kappa^{(1)}}^{\bp-1}(u^+/\Vert u^+\Vert)}
{\Vert u^+ \Vert^{\frac{\bp-3}{2}}\Vert u^- \Vert^{\frac{\bm-3}{2}}}
\int_0^\infty t_+^{2A+\frac{\bm+3}{2}}
J_{k_1+\frac{\bp-3}{2}}(2\pi\Vert u^+\Vert \tp)
      K_{k_1+\frac{\bp-3}{2}}(4\pi \tp)\\
&\qquad\qquad\times \int_0^1 
s^{\frac{\bm-1}{2}}(1-s^2)^A 
J_{\frac{\bm-3}{2}}(2\pi\Vert u^-\Vert \tp s)
dsd\tp.
\end{align*}
We then apply \cite[(6.567.1)]{GR}:
\begin{align*}
&\int_0^1 s^{\nu+1}(1-s^2)^\mu J_\nu(bs)ds=2^\mu \Gamma(\mu+1)
b^{-(\mu+1)}J_{\nu+\mu+1}(b)\\&\hspace{4cm}
(b>0,\ 
\mathrm{Re}(\nu)>-1,\ \mathrm{Re}(\mu)>-1)
\end{align*}
to compute the inner integral in $s$.
It yields that
\begin{align*}
I^\kappa(u)&=\frac{2^A(2\pi)^{1-A}i^{k_1}\Gamma(A+1)}
{\Vert u^+\Vert^{\frac{\bp-3}{2}}
\Vert u^-\Vert^{A+\frac{\bm-1}{2}}}
h_{\kappa^{(1)}}^{\bp-1}\left(\frac{u^+}{\Vert u^+\Vert}\right)\\
&\times\int_0^\infty t_+^{A+\frac{\bm+1}{2}}
J_{A+\frac{\bm-1}{2}}(2\pi\Vert u^-\Vert \tp)
J_{k_1+\frac{\bp-3}{2}}(2\pi\Vert u^+\Vert \tp)
K_{k_1+\frac{\bp-3}{2}}(4\pi \tp)d\tp.
\end{align*}
Finally, the identity \cite[(6.578.2)]{GR}
\begin{align*}
&\int_0^\infty t^{\rho-1}J_\lambda(at)J_\mu(bt)K_\nu(ct)dt\\
=&~\frac{2^{\rho-2}a^\lambda b^\mu c^{-\rho-\lambda-\mu}}
{\Gamma(\lambda+1)\Gamma(\mu+1)}
\Gamma\left(\frac{\rho+\lambda+\mu-\nu}{2}\right)
\Gamma\left(\frac{\rho+\lambda+\mu+\nu}{2}\right)\\
&\qquad\times F_4\left(\frac{\rho+\lambda+\mu-\nu}{2},
\frac{\rho+\lambda+\mu+\nu}{2};
\lambda+1,\mu+1;-\frac{a^2}{c^2}, -\frac{b^2}{c^2}\right)\\
&\hspace{4cm}(\mathrm{Re}(\rho+\lambda+\mu)>|\mathrm{Re}(\nu)|,\ 
\mathrm{Re}(c)>|\mathrm{Im}(a)|+|\mathrm{Im}(b)|)
\end{align*}
concludes the first assertion.

On the other hand when $\bm=1$, hence $V_1^-=\{0\}$, 
we see that
\begin{align*}
I^\kappa(u)=\int_{t>0}t^{2A-\frac{\bp-3}{2}}
K_{k_1+\frac{\bp-3}{2}}(4\pi t)
\int_{\sigma\in S(V_1)}
\be(t(\sigma,u))h_{\kappa^{(1)}}^{\bp-1}(\sigma)
t^{\bp-2}d\sigma dt.
\end{align*}
Again by (\ref{sphericalfourier}) we get 
\begin{align*}
I^\kappa(u)=\frac{2\pi i^{k_1}}{\Vert u\Vert^{\frac{\bp-3}{2}}}
h_{\kappa^{(1)}}^{\bp-1}\left(\frac{u}{\Vert u\Vert}\right) 
\int_0^\infty t^{2A+1}J_{k_1+\frac{\bp-3}{2}}(2\pi\Vert u\Vert t)
K_{k_1+\frac{\bp-3}{2}}(4\pi t)dt.
\end{align*}
Then the formula \cite[(6.576.3)]{GR}:
\begin{align*}
\int_0^\infty x^{-\lambda}K_\mu(ax)&J_\nu(bx)dx
=\frac{b^\nu\Gamma\left(\frac{\nu-\lambda+\mu+1}{2}\right)
\Gamma\left(\frac{\nu-\lambda-\mu+1}{2}\right)}
{2^{\lambda+1}a^{\nu-\lambda+1}\Gamma(\nu+1)}\\
&\times {}_2F_1\left(\frac{\nu-\lambda+\mu+1}{2},
\frac{\nu-\lambda-\mu+1}{2}; \nu+1; -\frac{b^2}{a^2}\right)\\
&\left(\mathrm{Re}(a\pm ib)>0,\ \mathrm{Re}(\nu-\lambda+1)>
|\mathrm{Re}(\mu)|\right),
\end{align*}
concludes the second assertion.
\end{proof}
\begin{proposition}\label{inversetransform}
Assume $\bm\ge 1$.
Then $I^\kappa(u)$ $(\bp>2)$, or $I^{\pm k}(u)$ $(\bp=2)$,
equals
$$
\frac{i^{k_1}\Gamma(k-k_1+1)\Gamma(2k_1+\bp-2)
\Gamma\bigl(k+\frac{\bp-\bm}{2}\bigr)}
{2^{k_1+\frac{\bp+\bm}{2}}(2\pi)^{2k+\bp-\frac{\bm+3}{2}}
\Gamma\bigl(k_1+\frac{\bp-1}{2}\bigr)}f(wn_0(u); h_\kappa^\bp),
$$
where $k=k_1$ understood in the case that $\bp=2$.
\end{proposition}
\begin{proof}
When $\bp>2$ and $\bp=1$, we apply \cite[Vol.I, p.64, (22)]{Er}:
\begin{align*}
{}_2 F_1(a,b;c;z)=(1-z)^{-a}{}_2F_1\left(a, c-b ; c;\frac{z}{z-1}\right)
\end{align*}
to rewrite the expression of $I^\kappa(u)$ in Lemma \ref{inversefourier2} (ii).
Similarly when $\bp>2$ and $\bm>1$, we 
apply \cite[Lemma 5.7]{KO}:
\begin{align*}
&F_4\left(\frac{p-1}{2},\frac{p+q-4}{2};\frac{p-1}{2},\frac{q-1}{2};
-\frac{|z'|^2}{4},-\frac{|z''|^2}{4}\right)\\
&=\tau(z',z'')^{-\frac{p+q-4}{2}}{}_2F_1\left(\frac{q-p}{4},\frac{p+q-4}{4},
\frac{q-1}{2};\frac{|z''|^2}{\tau(z',z'')^{2}}\right)
\end{align*}
to rewrite the expression of $I^\kappa(u)$ in Lemma \ref{inversefourier2} (i).
Then, in the case where $\bp>2$ and $\bm\ge 1$,
the assertion is obtained from the above rewrite with additional use of
\begin{align*}
C_n^\lambda(x)
&=\frac{\Gamma(n+2\lambda)}{\Gamma(n+1)\Gamma(2\lambda)}
{}_2F_1\left(-n, n+2\lambda;\lambda+\frac{1}{2};\frac{1}{2}-\frac{x}{2}\right)\\
&=\frac{\Gamma(n+2\lambda)}{\Gamma(n+1)\Gamma(2\lambda)}
{}_2F_1\left(-\frac{n}{2}, \frac{n}{2}+\lambda; \lambda+\frac12; 
1-x^2\right)
\end{align*}
(see \cite[Vol.I, p.175, 3.15.1 (3)]{Er} and \cite[Vol.I, p.111, 2.11 (2)]{Er})
and Lemma \ref{besselsection}.

For the remaining case that $\bp=2$ and $\bm\ge 1$ we use
the expression (\ref{J2}) of $J^{\pm k}(\lambda)$ and 
see that 
$$
\int_{\substack{\lambda\in V_1, Q(\lambda)>0,\\
(\pm\lambda,e_2)<0}}
\Vert \lambda\Vert^{2k-\bm}\exp\bigl((\pm\lambda,ye_2)\bigr)d\lambda
=2^{2k-1}\pi^{\frac{\bm}{2}-1}
\Gamma(k)\Gamma\left(k-\frac{\bm}{2}+1\right)(ye_2, ye_2)^{-k}
$$
\cite[\S 5, Hilfssatz 1]{Si} for $y>0$.
This yields that
$$
I^{\pm k}(u)=
\frac{i^k\Gamma(k)\Gamma(k-\frac{\bm}{2}+1)}
{2^{k+\frac52}\pi^{2k-\frac{\bm}{2}+1}}
\left(e_2\mp i\frac{u}{2}, e_2\mp i\frac{u}{2}\right)^{-k},\quad u\in V_1,
$$
\cite[Section 1, (1.22)]{Sh82}, 
which completes the proof of the proposition.
\end{proof}
\begin{corollary}\label{exactcharacterization} 
The function (\ref{archimedeancoefficient})
equals up to a constant multiple the value of a 
Bessel integral (\ref{besselintegral}) 
attached to the minimal $\OO(V^+)\times\OO(V^-)$-type
$\H^k(V^+)\boxtimes\mathbf{1}$ of the irreducible submodule
$\Pi^{\bp,\bm}_{k,0}$ of $I_{\varepsilon}(k+\bp-2)$.
\end{corollary}
\begin{proof} By definition we can check that
$$
W_\lambda(g_z; h_\kappa^\bp)=\be((\lambda,u))
\psi^\pm_{k+\bp-2}\bigl(m_0(a^{-1}, g_1)\bigr)\vert a\vert^{b-2}
W_{a\cdot g_1^{-1}\lambda}(1_b; h_\kappa^\bp)
$$
for every $g_z=n_0(u)m_0(a,g_1)\in P_z$ with
the unit element $1_b=1_{\bp+\bm}\in \OO(V)$.
However, Proposition \ref{inversetransform} and the Fourier inversion
theorem imply that $W_{a\cdot g_1^{-1}\lambda}(1_b; h_\kappa^\bp)$ 
equals $J^\kappa(a\cdot g_1^{-1}\lambda)$ 
up to a constant multiple.
Thus $W_\lambda(g_z; h_\kappa^\bp)$ equals a multiple of
(\ref{archimedeancoefficient}), which completes the proof.
\end{proof}

\subsection{Bessel integrals on the Levi subgroup $M$}
Take $z_0$ and $z_0^\ast\in V$ as in (\ref{standardisotropic}) and
also take the corresponding maximal $\Q$-parabolic subgroup $P=MN$ of $\OO(V)$
with the isomorphisms $n_0: V_1\simeq N$ and 
$m_0: \GL(\langle z_0\rangle)\times \OO(V_1)\simeq M$,
(\ref{explicitleviisomorphismV}) and (\ref{explicitleviisomorphismVm}).
Recall the decomposition $V=V_1\oplus\langle z_0, z_0^\ast\rangle$.
In this subsection, we assume that $\bp>2$ and $\bm\ge 2$ for $V$,
and hence $V_1$ is indefinite.
Let then $z_1$ and $z_1^*\in V_1$ be isotropic vectors given by
\begin{align*}
z_1=\frac12(e_2+e_{b-1}),\quad z_1^*=e_2-e_{b-1},
\end{align*}
which provide the decompositions
$$
V_1=\langle z_1\rangle
\oplus V_2\oplus\langle z_1^*\rangle\hbox{ and }
V=\langle z_0\rangle\oplus\langle z_1\rangle
\oplus V_2\oplus\langle z_1^*\rangle\oplus\langle z_0^*\rangle.
$$
Let $Q_1$ be the maximal $\Q$-parabolic subgroup of $\OO(V_1)$
stabilizing the line $\langle z_1\rangle$.
Its unipotent radical $N_1$ is isomorphic to $V_2$.
We write the isomorphism as $n_1: V_2\simeq N_1$, 
which is defined in the same manner as in (\ref{explicitleviisomorphismV}).

Let $w_{01}\in \OO(V)$ be the element that exchanges $z_0$ and $z_1$, 
$z_0^*$ and $z_1^*$, and also $w_{01}\vert_{V_2}
=\mathrm{id}_{V_2}$.
Besides it, let $w_1\in \OO(V_1)$ be 
the element that exchanges $z_1$ and $z_1^\ast$ and 
$w_1\vert_{V_2}=\mathrm{id}_{V_2}$.
Now for any $\lambda_2\in V_2$, $h(x)\in \H^{k,0}(V)$, and $g\in \mathrm{O}(V)$,
we define
\begin{align}\label{degeneratebesselintegral}
W_{\lambda_2}^{dg}(g; h)=\int_{V_2}\int_{\R}\be(-(\lambda,u))
f(w_{01}n_0(te_{b-1})m_0(1, w_1n_1(u))g; h)dtdu.
\end{align}
\begin{proposition}\label{lowerrankbessel}
Suppose $\bp>2$ and $\bm\ge 2$.
For every $\lambda_2\in V_2$ satisfying $Q(\lambda_2)>0$ and
$h_\kappa^\bp(x)\in \H^{k,0}(V)$, one has the following identities.

\medskip
\noindent
{\rm (i)} If $k>k_1$ for $\kappa$, then 
$W^{dg}_{\lambda_2}(1_b; h_\kappa^\bp)=0$.

\medskip
\noindent
{\rm (ii)} If $k_1=k$ for $\kappa$, then 
$$
W^{dg}_{\lambda_2}(1_b; h_\kappa^\bp)=
\frac{2\sqrt{\pi}\Gamma(k+\frac{\bp-3}{2})}{\Gamma(k+\frac{\bp-2}{2})}
\int_{V_2}\be(-(\lambda_2,u))
f_1(w_1n_1(u); h_{\kappa^{(1)}}^{\bp-1})du,
$$
where 
$$
f_1(g_1; h_{\kappa^{(1)}}^{\bp-1})=h_{\kappa^{(1)}}^{\bp-1}
((g_1^{-1}z_1)^+)\Vert (g_1^{-1}z_1)^+\Vert^{-2k-\bp+3},\quad
g_1\in \OO(V_1)\quad  (k=k_1).
$$
Thus, up to a constant multiple,
$W^{dg}_{\lambda_2}(1_b; h_\kappa^\bp)$ equals the value
of Bessel integral on $\OO(V_1)$ of type (\ref{besselintegral})
at the unit element $1_{b-2}$.
\end{proposition}
\begin{proof}  
By definition, $f(g; h_\kappa^\bp)$, $g\in \OO(V)$, can be written as
$$
\Vert (g^{-1}z_0)^+\Vert^{-k-k_1-\bp+2}h_{\kappa^{(1)}}^{\bp-1}
((g^{-1}z_0)^+)C_{k-k_1}^{k_1+\frac{\bp-2}{2}}
\left(\frac{(g^{-1}z_0,e_1)}{\Vert(g^{-1}z_0)^+\Vert}\right).
$$
Note that
$$
[w_{01}n_0(te_{\bp+\bm-1})m_0(1, w_1n_1(u))]^{-1}z_0=n_1(u)^{-1}w_1z_1-tz_0
$$
with $n_1(u)^{-1}w_1z_1\in V_1$, 
and that
$h_{\kappa^{(1)}}^{\bp-1}((x-tz_0)^+)
=h_{\kappa^{(1)}}^{\bp-1}(x^+)$ for $t\in \R$ and $x\in V_1$.
These are combined to give that
\begin{align*}
&W^{dg}_{\lambda_2}(1_b; h_\kappa^\bp)=
\int_{V_2}\be(-(\lambda_2,u))h_{\kappa^{(1)}}^{\bp-1}
((n_1(u)^{-1}w_1z_1)^+)\\
&\times
\int_{\mathbb{R}}\Vert (n_1(u)^{-1}w_1z_1-tz_0)^+\Vert^{-k-k_1-\bp+2}
C_{k-k_1}^{k_1+\frac{\bp-2}{2}}
\left(\frac{-t}{2\Vert n_1(u)^{-1}w_1z_1-tz_0\Vert}\right)dtdu
\end{align*}
By changing the variables from $-t/(2\Vert n_1(u)^{-1}w_1z_1-tz_0\Vert)$ 
to $s$,
the inner integral in $t$ can be rewritten as
$$
2\Vert (n_1(u)^{-1}w_1z_1)^+\Vert^{-k-k_1-(\bp-1)+2}
\int_{-1}^1(1-s^2)^{\frac{k+k_1+\bp-5}{2}}C_{k-k_1}^{k_1+\frac{\bp-2}{2}}(s)ds.
$$
When $k-k_1$ is odd, this becomes $0$, 
since its integrand is an odd function with respect to $s$.
When $k-k_1$ is even and $k>k_1$, it is again shown to
vanish by using the partial integration \cite[p.100, line 1]{MO}:
$$
\int_{-1}^1(1-s^2)^{\nu-\frac12}C_n^\nu(s)F(s)dt
=C\int_{-1}^1(1-s^2)^{n+\nu-\frac12}\frac{d^nF}{dt^n}(s)dt
$$
with a constant $C=C_{\nu,n}$.
Finally when $k=k_1$,
$W_{\lambda_2}^{dg}(1_b; h_\kappa^\bp)$ is equal to
$$
2\int_{V_2}\be(-(\lambda_2,u))
\frac{h^{\bp-1}_{\kappa^{(1)}}((n_1(u)^{-1}w_1z_1)^+)}
{\Vert (n_1(u)^{-1}w_1z_1)^+\Vert^{k+k_1+(\bp-1)-2}}du
\int_{-1}^{1}(1-s^{2})^{k+\frac{\bp-5}{2}}ds,
$$
which concludes the proposition.
\end{proof}
Let $I_\varepsilon(k+\bp-3)=I_\varepsilon(k+(\bp-1)-2)$ be 
a degenerate principal series of $\OO(V_1)$, which then
contains the irreducible submodule $\Pi^{\bp-1,\bm-1}_{k,0}$ 
with the minimal $\OO(V_1^+)\times\OO(V_1^-)$-type 
$\H^{k,0}(V_1)=\H^k(V_1^+)\boxtimes\mathbf{1}$.
\begin{corollary} The function
$W_{\lambda_2}^{dg}(m_0(1,g_1); h_\kappa^\bp)$, $g_1\in \OO(V_1)$,
equals a multiple of a Bessel integral on $\OO(V_1)$ of type 
(\ref{besselintegral})
attached to the minimal $\OO(V_1^+)\times\OO(V_1^-)$-type $\H^{k,0}(V_1)$
of the irreducible $\OO(V_1)$-module $\Pi^{\bp-1,\bm-1}_{k,0}$.
In particular the Fourier expansion of the constant term
$\delta_{kk_1}|a|\cdot\vartheta_{L_1}(f_{L_1};h_{\kappa^{(1)}}^{\bp-1})(g_1)/
(\sqrt{2}\Vert v_0^+(z)\Vert)$ in Theorem \ref{fourierexpansion>2} 
involves pulling back the integrals
$W_{\lambda_2}^{dg}(m_0(a,g_1); h_\kappa^\bp)$,
$\lambda_2\in L_2'$, $Q(\lambda_2)>0$, by an isometry $v_0$.
\end{corollary}
\begin{proof}
The assertions are direct consequences of Proposition \ref{lowerrankbessel}
and the definition (\ref{degeneratebesselintegral}), 
and also of Proposition \ref{inversetransform} and 
Theorem \ref{fourierexpansion>2}
applied to the factor 
(\ref{l1lift}) occurring in the constant term of 
$\vartheta_L(f; h_\kappa^\bp)(g_z)$.
\end{proof}

\section{A restriction}
Zemel \cite{Zeb} investigated restrictions of theta liftings of 
$f(\tau)\in S_\nu(D_L)$ to $\mathrm{SO}(\bp,\bm)$ to subgroups and 
explicitly described them as liftings of the {\it theta contractions} 
of $f(\tau)$ with the appropriate theta functions
associated with sublattices in $L$ (see \cite[Theorem 2.6]{Zeb}).
Here we can directly compute such a restriction using our Fourier
expansion formula of $\vartheta_L(f)(g)$, Theorem \ref{fourierexpansion>2}.
In this section we will compare these two computations 
for the lifting to $\OO(3,2)$ and check our result is consistent with 
\cite{Zeb}.
 
Let $L$ be an even lattice of signature $(3,2)$.
Take any pair of primitive isotropic 
vector $z\in L$ and an element $z'\in L'$ such that $(z,z')=1$.
For every such pair we construct a sublattice $K$ of $L$ in the following way.
Set $L_1=L\cap \langle z\rangle^\perp\cap \langle z'\rangle^\perp$ as before,
and let $K_1$ be a positive definite sublattice of $L_1$ of rank $2$.
Then we set an even lattice
\begin{equation}\label{3,1lattice}
K=L\cap (K_1\oplus\langle z,z'\rangle),
\end{equation}
of signature $(3,1)$.
According to this definition $\OO(K_\R)$ will be identified 
with the corresponding subgroup of $\OO(L_\R)$. 
Then the maximal $\Q$-parabolic subgroup $P'_z=M_z'N_z'$ of 
$\OO(K_\R)$, which stabilizes $\langle z\rangle\subset K_\R$,
is also embedded into the maximal $\Q$-parabolic subgroup $P_z$
of $\OO(L_\R)$ in the canonical way.
As to the isomorphisms $n'_z: (K_1)_\R\simeq N_z'$ and 
$m_z': \GL(\langle z\rangle)\times \OO((K_1)_\R)\simeq M_z'$,
which are defined in the same way as (\ref{bijectionlevi1}) and 
(\ref{bijectionlevi2}),
we will regard that $n'_z(u')=n_z(u')$ and
$m_z'(a,g'_1)=m_z(a,g'_1)$ for every $u'\in (K_1)_\R$, and $a\in \R^\times$ and 
$g'_1\in\OO((K_1)_\R)\simeq \OO(2)$ through this embedding $P_z'\subset P_z$.
 
Under the above setup the restriction of 
$\vartheta_L(f)(g)$ on $\OO(L_\R)$ to $\OO(K_\R)$ gives 
an automorphic form for the discrete group 
$\varGamma_L\cap \OO(K_\R)$.
We remark that 
$\varGamma_K\subset \varGamma_L\cap \OO(K_\R)$ since 
$\gamma\lambda-\lambda\in K\subset L$ for 
every $\gamma\in \varGamma_K$ and $\lambda\in L'$.
In the following we will especially focus on the case 
where $K_1\subset L_1$ is primitive and $(z,L)=(z,K)=N\Z$.
Under these assumptions we can write $K$ as $K=K_1\oplus\Z z\oplus\Z\zeta'$
as an additive group with $\zeta'\in K$ such that $(z,\zeta')=N$.
This implies that $K\subset L$ is primitive, and furthermore that 
$\varGamma_L\cap \OO(K_\R)=\varGamma_K$.

We will use Theorem \ref{fourierexpansion>2} to describe the restriction 
of $\vartheta_L(f)(g)$ to $\OO(K_\R)$ directly 
in terms of the Fourier series expansion.
For that purpose we choose an isometry $v_0: L_\R\simeq \R^{3,2}$
so that its restriction to $K_\R$ is compatible with the decomposition
$K_\R=(K_1)_\R\oplus\langle z, z'-q(z')z\rangle$; more precisely we set
\begin{align*}
&v_0\left(\frac{z+z'-q(z')z}{2}\right)=e_1,\quad
v_0((K_1)_\R)=\langle e_2, e_3\rangle, \\
&v_0(L_1(\R))=v_0((K_1)_\R)\oplus\langle e_4\rangle,\quad
v_0\left(\frac{z-z'+q(z')z}{2}\right)=e_5,
\end{align*}
where $\{e_1\}_{i=1}^5$ denotes the standard ordered orthogonal basis
of $\R^{3,2}$.
Then we set an isometry $w_0=w_0^++w_0^-: K_\R\simeq \R^{3,1}$ 
as the restriction of $v_0=v_0^++v_0^-$ to $K_\R$.
Lastly we let $w_0$ restrict to the isometry $w_1$ of $(K_1)_\R$ onto
$\R^{2,0}$.
Now we get the following:
\begin{proposition}\label{Fourierexpansionofrestriction}
Assume the above setting.
Let $h_\kappa^3(x)\in\H^{k,0}(\R^{3,2})$ with
$\kappa=(k,\pm k_1)$ and $g_K=n'_z(u')m_z'(a,g'_1)\in P'_z$ with 
$u'\in (K_1)_\R$, $a\in\R^\times$, and $g_1'\in \OO((K_1)_\R)$. 
Then 
\begin{align}\label{restriction}
\vartheta_L(f; h_\kappa^3)&(g_K)=
\frac{\delta_{kk_1}|a|}{\sqrt{2}\Vert w_0^+(z)\Vert}
\bigl(f_{L_1}(\tau), \Theta_{L_1}(\tau, g_1'; h_{\pm k}^2)\bigr)_\tau \\
&+\sum_{0\neq\lambda\in K'_1}
\biggl(\sum_{n|\lambda}n^k
\sum_{\substack{\gamma\in L_1', n|\gamma \\ \gamma|_{K_1}=\lambda}}
\sum_{\substack{\delta\in M'/L\\ \pi(\delta)=\gamma/n}}
\be((n\delta,z'))
c\left(\frac{q(\gamma)}{n^2},\delta\right)\biggr) 
W_\lambda^\kappa(g_K),\notag
\end{align}
where 
$\gamma|_{K_1}=\lambda$ means that $(\gamma,\nu)=(\lambda,\nu)$ 
for all $\nu\in K_1$ and
\begin{align*}
W^\kappa_\lambda(g_K)=&
\frac{(-i)^{k-k_1}(a/|a|)^{k-k_1}|a|^{k+1}}
{2^{k_1-2}(k-k_1)!\Vert w_0^+(z)\Vert^{k+1}}\\
\times &\be((\lambda,u'))
\frac{\bigl(w_1((g_1')^{-1}\lambda), e_2\pm ie_3\bigr)^{k_1}}
{\Vert w_1(\lambda)\Vert^{k_1}}
K_{k_1}\left(\frac{2\pi|a|\Vert w_1(\lambda)\Vert}
{\Vert w_0^+(z)\Vert}\right)
\end{align*}
for every $\lambda\in K_1'\setminus\{0\}$.
This restriction is a modular form on $\OO(K_\R)\simeq\OO(3,1)$
for the group $\varGamma_K$, 
which belongs archimedean locally
to the $\OO(3,1)$-module $\Pi^{3,1}_{k,0}$.
Moreover its constant term involves the values on $\OO((K_1)_\R)\subset
\OO(L_\R)$ of cuspidal automorphic forms belonging to the discrete series 
representation $\Pi^{2,1}_{k,0}$ of $\OO(2,1)$.
\end{proposition}
\begin{remark}
In (\ref{restriction}) we have ignored the condition $q(\gamma)>0$ 
for the summation in $\gamma\in L_1'$ with $n|\gamma$ and 
$\gamma|_{K_1}=\lambda$, since the coefficients $c(n,\delta)=0$ for $n\leq 0$.
Also we note that $q(\lambda)>0$ for every nonzero $\lambda\in K_1'$.
\end{remark}
Now we specialize \cite[Theorem 2.6]{Zeb} to give the other description of 
$\vartheta_L(f;h_\kappa^3)(g_K)$.
Take the orthogonal decomposition $L_\Q=K_\Q\oplus K_\Q^\perp$, and write 
$x=x_K+x_{K^\perp}\in L_\Q$ according to it.
Then the orthogonal projection from $L_\Q$ to $K_\Q$ sends
$L'$ to $K'$, and its composition with $K'\to K'/K$ factors
through $L'/K$ to define the map $p_K: L'/K\to K'/K=D_K$.
Define a $\C[D_L]\otimes\C[D_K]$-valued theta function on 
$\mathbb{H}_1$ by
$$
\Theta_{L,K}(\tau)=
\sum_{\delta\in D_K}\sum_{\substack{\lambda\in L'/K\\ p_K(\lambda)=\delta}}
\be(\overline{\tau} q\bigl(\lambda_{K^\perp})\bigr)\mathfrak{e}_{\lambda+L}
\otimes\mathfrak{e_{\delta}},
$$
where $\lambda_{K^\perp}\in K_\Q^\perp$ is well-defined for $\lambda\in L'/K$.
Then the following states a specialization of \cite[Theorem 2.6]{Zeb} 
to our case:
\begin{theorem}\label{zemelstheorem}
Let $f(\tau)\in S_\nu(D_L)$.
Then the function 
\begin{equation}\label{thetacontraction}
\Theta_{(L,K)}(f;\tau):\tau\mapsto \langle f(\tau),\Theta_{L,K}(\tau)\rangle_L
\end{equation}
belongs to $S_\nu(D_K)$ and one has the identity
\begin{equation}\label{seesaw}
\vartheta_L(f;h_\kappa^3)(g_K)=
\bigl(\Theta_{(L,K)}(f;\tau),\Theta_K(\tau,g_K;h_\kappa^3)\bigr)_\tau
\end{equation}
for every $g_K\in \OO(K_\R)$.
\end{theorem}
As in (\ref{ellipticfourier}) for $f(\tau)$
we write the Fourier expansion of (\ref{thetacontraction}) as
$$
\Theta_{(L,K)}(f;\tau)=\sum_{\delta\in D_K}
\sum_{\substack{m>0\\ m\in q(\delta)+\Z}}
C_{L,K}(m,\delta)\be(m\tau)\mathfrak{e}_\delta.
$$
Then the definition of $\Theta_{(L,K)}(f;\tau)$ implies that
\begin{equation}\label{coefficientsofthethetacontraction}
C_{L,K}(m,\delta)=\sum_{\substack{\mu\in L'/K\\ p_K(\mu)=\delta}}
c(m+q(\mu_{K^\perp}),\mu+L)
\end{equation}
for every $\delta\in D_K$ and $m>0$ with $m\in q(\delta)+\Z$. 
On the other hand Theorem \ref{fourierexpansion>2} gives that
the right-hand side of (\ref{seesaw}) is equal to
\begin{align}\label{fourierseesaw}
& \frac{\delta_{kk_1}|a|}{\sqrt{2}\Vert w_0^+(z)\Vert}
\bigl(\Theta_{L,K}(f,\tau)_{K_1},
\Theta_{K_1}(\tau,g_1';h_{\pm k}^2)\bigr)_\tau\\
&\quad +\sum_{0\neq\lambda\in K_1'}
\biggl(\sum_{n|\lambda}
n^k\sum_{\substack{\varepsilon\in F'/K\\ \pi_{K_1}(\varepsilon)=\lambda/n}}
\be((n\varepsilon, z'))C_{L,K}\left(\frac{q(\lambda)}{n^2},\varepsilon\right)
\biggr)W_\lambda^\kappa(g_K), \notag
\end{align}
where $F'=\{\varepsilon\in K'\mid (\varepsilon,z)=0\ \mathrm{mod}\ N\}$ 
and $\pi_{K_1}$ is the map from $F'/K$ to $D_{K_1}$ 
defined in the same way as $\pi$ in (\ref{prM'}). 
Here we also note that $z'\in K'$ by (\ref{3,1lattice}).
Thus our task is to check the equality of the right-hand side of
(\ref{restriction}) and (\ref{fourierseesaw}).
To this end we observe the following:
\begin{lemma}\label{basicidentity}
{\rm (i)} Suppose $0\neq \lambda\in K_1'$ and
$n|\lambda$, $n\in\mathbb{N}$ (i.e. $\lambda/n\in K_1'$).
Then one has 
\begin{align}\label{Fouriecoefficientofrestriction}
\sum_{\substack{\varepsilon\in F'/K\\ \pi_{K_1}(\varepsilon)=\lambda/n}}
\be((n\varepsilon,z'))
C_{L,K}\left(\frac{q(\lambda)}{n^2},\varepsilon\right)
=\sum_{\substack{\gamma\in L_1', n|\gamma\\ \gamma|_{K_1}=\lambda}}
\sum_{\substack{\delta\in M'/L\\ \pi(\delta)= \gamma}}
\be((n\delta,z'))c\left(\frac{q(\gamma)}{n^2},\delta\right).
\end{align}

\medskip
\noindent
{\rm (ii)} For $f(\tau)\in S_\nu(D_L)$ one has 
that $\Theta_{L,K}(f,\tau)_{K_1}=\Theta_{L_1,K_1}(f_{L_1},\tau)$.
\end{lemma}
\begin{proof}
We apply (\ref{coefficientsofthethetacontraction}) to
rewrite the left-hand side of (\ref{Fouriecoefficientofrestriction}) as
\begin{align*}
&\sum_{\substack{\varepsilon\in F'/K\\ \pi_{K_1}(\varepsilon)=\lambda/n}}
\be((n\varepsilon,z'))
\sum_{\substack{\mu\in L'/K\\ p_K(\mu)=\varepsilon}}
c\left(\frac{q(\lambda)}{n^2}+q(\mu_{K^\perp}), \mu+L\right)\\
&\qquad=\sum_{\substack{\mu\in M'/K\\\pi_{K_1}(p_K(\mu))=\lambda/n}}
\be((n\mu,z'))
c\biggl(\frac{q(\lambda+n\mu_{K^\perp})}{n^2},\mu+L\biggr).
\end{align*}
Then (\ref{Fouriecoefficientofrestriction}) is obtained,
if we apply the bijection 
\begin{align*}
&\{\mu\in M'/K\mid \pi_{K_1}(p_K(\mu))=\lambda/n\}\\
&\qquad \qquad \simeq 
\{(\gamma,\delta)\in L'_1\times M'/L\mid n|\gamma,\ \gamma|_{K_1}=\lambda,\ 
\pi(\delta)=\gamma/n\},
\end{align*}
sending $\mu$ to $(\lambda+n\mu_{K_\R^\perp}, \mu+L)$
for a given $\lambda\in K_1'$ with $n|\lambda$, to
rewrite the right-hand side of the identity above.

The same bijection in the case where $n=1$ is 
again applied to give the equality
\begin{equation}\label{zemeln=1}
\sum_{\substack{\varepsilon\in F'/K\\ \pi_{K_1}(\varepsilon)=\lambda}}
\sum_{\substack{\mu\in L'/K\\ p_K(\mu)=\varepsilon}}
c(m+q(\mu_{K_\R^\perp}),\mu)=\sum_{\substack{\gamma\in L_1'\\ 
\gamma|_{K_1}=\lambda}}\sum_{\substack{\delta\in M'/L\\ \pi(\delta)=\gamma}}
c(m+q(\gamma_{(K_1)_\R^\perp}), \delta)
\end{equation}
for $\lambda\in K_1'$ and $m\in q(\lambda)+\Z$.
This shows the claimed equality in (ii),
since each side of  (\ref{zemeln=1}) equals the $(m,\lambda)$-th 
Fourier coefficient of 
$\Theta_{L,K}(f,\tau)_{K_1}$ or $\Theta_{L_1,K_1}(f_{L_1},\tau)$
respectively.
Thus the proof is finished.
\end{proof} 

By putting (\ref{Fouriecoefficientofrestriction}) into (\ref{fourierseesaw}),
we see that (\ref{fourierseesaw}) and (\ref{restriction}) coincide for 
every nonzero $\lambda\in K_1'$.
As to the constant terms 
Lemma \ref{basicidentity} (ii) and \cite[Theorem 2.6]{Zeb}
applied to the pairing
 $(\Theta_{L_1,K_1}(f_{L_1},\tau), 
\Theta_{K_1}(\tau, g_1'; h_{\pm k}^2))_\tau$ conclude their coincidence.
Thus we have checked the equality of the right-hand side of (\ref{restriction})
and (\ref{fourierseesaw}) as desired.

\end{document}